\documentclass{amsart}
\calclayout
\usepackage{amssymb,amsmath,amsfonts,epsfig,latexsym,tikz}
\usepackage{tikz-cd}
\usepackage{hyperref}
\usepackage{enumerate}
\usepackage{mathtools}
\usepackage{verbatim}
\usepackage{cleveref}
\usepackage[shortlabels]{enumitem}
\usepackage{subcaption}
\usepackage{stmaryrd}

\usetikzlibrary{positioning}
\usetikzlibrary{matrix}
\usetikzlibrary{decorations}
\usetikzlibrary{decorations.pathreplacing, decorations.pathmorphing, angles,quotes}

\newtheorem{theorem}{Theorem}[section]

\newtheorem{proposition}[theorem]{Proposition}
\newtheorem{lemma}[theorem]{Lemma}
\newtheorem{corollary}[theorem]{Corollary}
\newtheorem{conjecture}[theorem]{Conjecture}

\theoremstyle{definition}
\newtheorem{remark}[theorem]{Remark}

\newtheorem{example}[theorem]{Example}

\def\R{\mathbb{R}}

\def\Z{\mathbb{Z}}

\def\1{\mathbf{1}}

\def\<{\langle}
\def\>{\rangle}

\DeclareMathOperator{\codim}{codim}

\DeclareMathOperator{\ord}{ord}

\DeclareMathOperator{\Star}{Star}

\DeclareMathOperator{\spn}{span}

\DeclareMathOperator{\ev}{ev}

\DeclareMathOperator{\prim}{prim}
\DeclareMathOperator{\Char}{char}

\DeclareMathOperator{\sgn}{sgn}
\DeclareMathOperator{\rank}{rank}

\newcommand\nullset\varnothing

\begin{document}
	
	\title{Determinants of Hodge--Riemann forms}

		\author{Matt Larson, Isabella Novik and Alan Stapledon}      
	
	\address{Princeton University and the Institute for Advanced Study}
\email{mattlarson@princeton.edu}

\address{University of Washington}
\email{novik@uw.edu}

\address{Sydney Mathematics Research Institute}
\email{astapldn@gmail.com}

	\begin{abstract}
We calculate the determinant of the bilinear form in middle degree of the generic artinian reduction of the Stanley--Reisner ring of an odd-dimensional simplicial sphere. This 
proves the odd multiplicity conjecture of Papadakis and Petrotou and implies that this determinant is a complete invariant of the simplicial sphere. 
We extend this result to odd-dimensional connected oriented simplicial homology manifolds. In characteristic $2$, we prove a generalization to  
the Hodge--Riemann forms 
of any connected 
simplicial homology manifold.
To prove the latter theorem, we establish the strong Lefschetz property for certain quotients of the Stanley--Reisner rings of connected simplicial pseudomanifolds. 
	\end{abstract}
	
		\maketitle
	
	\vspace{-20 pt}
	
	\setcounter{tocdepth}{1}

	\section{Introduction}

This paper provides several far-reaching generalizations of the algebraic $g$-conjecture for oriented simplicial pseudomanifolds. Its proof was announced in \cite{APP}; see also \cite{Adiprasitog,KaruXiaoAnisotropy,PapadakisPetrotoug}. Our starting point is the odd-multiplicity conjecture of Papadakis and Petrotou \cite[Conjecture 5.4]{PapadakisPetrotouGenericAnisotropy1Spheres}. Below, after setting up the notation, we summarize our main results.

Let $\Delta$ be a simplicial complex  of dimension $d-1 > 0$ with vertex set $V = \{1, \dotsc, n\}$. Let $k$ be a field, and set $K = k(a_{i,j})_{1 \le i \le d, \, 1 \le j \le n}$ to be a purely transcendental field extension of $k$. We assume that $\Delta$ is a connected homology manifold over $k$, i.e., $\Delta$ is connected, and the link of every nonempty face $G$ of $\Delta$ has the same homology as a sphere of dimension $d - |G| - 1$ over $k$.  Let $K[\Delta]$ be the Stanley--Reisner ring of $\Delta$, and set $\theta_i = a_{i,1} x_1 + \dotsb + a_{i,n} x_n \in K[\Delta]$ for $i \in \{1, \dotsc, d\}$, so $\theta_1, \dotsc, \theta_d$ is a linear system of parameters for $K[\Delta]$. Let $H(\Delta) = K[\Delta]/(\theta_1, \dotsc, \theta_d)$ be the generic artinian reduction of $K[\Delta]$.

Assume that $\Delta$ is oriented. Then there is a distinguished isomorphism $\deg \colon H^{d}(\Delta) \to K$ \cite{BrionStructurePolytopeAlgebra}; see Section~\ref{sec:2}. Let $\overline{H}(\Delta)$ be the Gorenstein quotient of $H(\Delta)$, i.e., 
the quotient 
by the ideal
 $(y \in H(\Delta) : (y  z)_d = 0 \text{ for all }z \in H(\Delta))$,
where $w_d$ denotes the degree $d$ component of $w$ in $H(\Delta)$. 
If $\Delta$ is a homology sphere over $k$, i.e., a homology manifold with the same homology over $k$ as a sphere of dimension $d - 1$, then $\overline{H}(\Delta) = H(\Delta)$. 
By construction, $\overline{H}(\Delta)$ is an artinian Gorenstein ring: for each $q$, the bilinear form $\overline{H}^q(\Delta) \times \overline{H}^{d-q}(\Delta) \to K$ given by $(y, z) \mapsto \deg(y z)$ is nondegenerate.

Suppose that $d$ is even. Let $D_{d/2} \in K^{\times}/(K^{\times})^2$ be the determinant of the nondegenerate bilinear form on $\overline{H}^{d/2}(\Delta)$. That is, 
choose a basis $y_1, \dotsc, y_p$ for $\overline{H}^{d/2}(\Delta)$, and let $M$ be the symmetric matrix whose $(i, j)$th entry is $\deg(y_i  y_j)$. Then $D_{d/2}$ is the image of $\det M$ in $K^{\times}/(K^{\times})^2$; choosing a different basis for $\overline{H}^{d/2}(\Delta)$ only changes $\det M$ by a square, so $D_{d/2}$ is well-defined. 
For a subset $F = \{j_1 < \dotsb < j_{d}\}$ of $V$ of size $d$, set $[F]$ to be the determinant of the $d \times d$ matrix whose $(i, m)$th entry is $a_{i,j_m}$.

\begin{theorem}\label{thm:middledegree}
	Let $d$ be even, and let $\Delta$ be a connected oriented simplicial $k$-homology manifold of dimension $d-1$.
	Then
		$$D_{d/2} = \lambda \prod_{F \text{ facet of }\Delta} [F] \in K^{\times}/(K^{\times})^2$$
		for some $\lambda \in k^{\times}/(k^{\times})^2$.

\end{theorem}

Papadakis and Petrotou proved Theorem~\ref{thm:middledegree} 
for $1$-dimensional simplicial spheres 
\cite[Proposition 5.1]{PapadakisPetrotouGenericAnisotropy1Spheres}.

Let $F$ be a subset of $V$ of size $d$. 
As $[F]$ is an irreducible polynomial (see Lemma~\ref{lemma:irreducible}), it defines a valuation $\ord_{[F]} \colon K^{\times} \to \mathbb{Z}$ given by the order of vanishing along the hypersurface defined by $[F]$. This descends to a homomorphism $\ord_{[F]} \colon K^{\times}/(K^{\times})^2 \to \mathbb{Z}/2\mathbb{Z}$. 
We immediately deduce the following corollary of Theorem~\ref{thm:middledegree}. 
It implies that the determinant of the bilinear form on $\overline{H}^{d/2}(\Delta)$ is a complete invariant of the connected orientable simplicial $k$-homology manifold $\Delta$. 
\begin{corollary}\label{cor:ordervanishing}
		Let $d$ be even, and let $\Delta$ be a connected oriented simplicial $k$-homology manifold of dimension $d-1$
		 with vertex set $V$. 
 Let $F$ be a subset of $V$ of size $d$. Then
	$$\ord_{[F]}(D_{d/2}) = \begin{cases} 1 & \text{if }F \text{ is a facet of }\Delta \\ 0 & \text{otherwise}.\end{cases}$$
\end{corollary}

When $\Delta$ is a simplicial sphere, 
Corollary~\ref{cor:ordervanishing} was conjectured by Papadakis and Petrotou \cite[Conjecture~5.4]{PapadakisPetrotouGenericAnisotropy1Spheres}, who called it the \emph{odd multiplicity conjecture}. This conjecture has motivated our work.

\medskip

We prove a generalization of the odd multiplicity conjecture. 
Assume that $\Char k = 2$, or $\Char k = 0$ and 
the integral homology of the link of any 
face (including the empty face)
of $\Delta$ has no $2$-torsion. 
Let $\ell = \sum_{j=1}^{n} x_j \in \overline{H}^1(\Delta)$. For $0 \le q \le d/2$, define the Hodge--Riemann form $\overline{H}^q(\Delta) \times \overline{H}^q(\Delta) \to K$ via $(y, z) \mapsto \deg(\ell^{d - 2q}  y  z)$. 
When $d$ is even and $q = d/2$, the Hodge--Riemann form is the bilinear form on $\overline{H}^{d/2}(\Delta)$ considered above.
 Let $D_{q}$ be the determinant of the Hodge--Riemann form on $\overline{H}^q(\Delta)$.

\begin{theorem}\label{thm:oddmultHodgeRiemann}
	Let $\Delta$ be a connected oriented simplicial $k$-homology manifold of dimension $d-1$ with vertex set $V$,  and let $0 \le q \le d/2$. 
	Assume that $\Char k = 2$, or $\Char k = 0$ and the integral homology of the link of any face (including the empty face) of $\Delta$ has no $2$-torsion.
	Let $F$ be a subset of $V$ of size $d$. Then
	$$\ord_{[F]}(D_{q}) = \begin{cases} 1 & \text{if }F \text{ is a facet of }\Delta \\ 0 & \text{otherwise}.\end{cases}$$
\end{theorem}

When $d$ is even and $q = d/2$, Theorem~\ref{thm:oddmultHodgeRiemann} follows from
Corollary~\ref{cor:ordervanishing}.
The nondegeneracy of the Hodge--Riemann form, which is part of 
Theorem~\ref{thm:oddmultHodgeRiemann},
is equivalent to the map $\overline{H}^q(\Delta) \to \overline{H}^{d  - q}(\Delta)$ given by multiplication by $\ell^{d  -2q}$ being an isomorphism. By Lemma~\ref{lem:slp}, this is equivalent to $\overline{H}(\Delta)$ having the strong Lefschetz property in degree $q$, i.e., that there is some $y \in \overline{H}^1(\Delta)$ such that the map $\overline{H}^q(\Delta) \to \overline{H}^{d  - q}(\Delta)$ given by multiplication by $y^{d - 2q}$ is an isomorphism.

In particular, 
Theorem~\ref{thm:oddmultHodgeRiemann}
is a generalization of the algebraic $g$-conjecture for $\Delta$ (that $\overline{H}(\Delta)$ has the strong Lefschetz property), 
and it implies that the Hodge--Riemann form in any degree is a complete invariant of $\Delta$. 
A proof of the algebraic $g$-conjecture for connected oriented simplicial pseudomanifolds was announced in \cite{APP}; see also \cite{Adiprasitog,KaruXiaoAnisotropy,PapadakisPetrotoug}. We have been heavily inspired by the recent progress on the algebraic $g$-conjecture, and, in particular, the key insight that one should  study the generic artinian reduction of $K[\Delta]$ and the corresponding degree map. See also \cite{APPVolume}.

Theorem~\ref{thm:oddmultHodgeRiemann}
follows from a strengthening of the algebraic $g$-conjecture for less generic artinian reductions of $K[\Delta]$. 
Let $F$ be a subset of $V$ of size $d$ which is not a facet of $\Delta$, and set $\theta_1^F = \sum_{j \not \in F} a_{1, j} x_j$. Then $\theta_1^F, \theta_2, \dotsc, \theta_d$ is still a linear system of parameters for $K[\Delta]$ 
(see Proposition~\ref{prop:stanleycriterion}). 
Let $H_F(\Delta) = K[\Delta]/(\theta_1^F, \theta_2, \dotsc, \theta_d)$. Then there is a distinguished isomorphism $\deg_F \colon H^d_F(\Delta) \to K$ (see Section~\ref{sec:2}). Set $\overline{H}_F(\Delta)$ to be the Gorenstein quotient of $H_F(\Delta)$. 
For example, if $\Delta$ is a homology sphere over $k$, then $H_F(\Delta) = \overline{H}_F(\Delta)$. We deduce Theorem~\ref{thm:oddmultHodgeRiemann} from the following theorem. 

\begin{theorem}\label{thm:strongg}
	Let $\Delta$ be a connected oriented simplicial 
	$k$-homology manifold 
	of dimension $d-1$, and let $0 \le q \le d/2$. 
	Assume that $\Char k = 2$, or $\Char k = 0$ and 
	the integral homology of the link of any face (including the empty face) of $\Delta$ has no $2$-torsion.
	Then 
	for every non-face $F$ of size $d$, $\overline{H}_F(\Delta)$ has the strong Lefschetz property in degree $q$. 
\end{theorem}

We use the characteristic $2$ method of Papadakis and Petrotou \cite{PapadakisPetrotoug}, as refined by Karu and Xiao \cite{KaruXiaoAnisotropy}. In the case of $\overline{H}(\Delta)$, this method was used to show that the Hodge--Riemann forms are \emph{anisotropic} 
\cite[Theorem 4.4]{KaruXiaoAnisotropy}.
This can fail in the case of $\overline{H}_F(\Delta)$ (see Example~\ref{ex:anisotropyfails}), but we show that anisotropy ``almost'' holds (Proposition~\ref{prop:almostanisotropy}) and use this to deduce Theorem~\ref{thm:strongg}.

When $k$ has characteristic $0$ and $\Delta$ is a polytopal sphere, i.e., is the boundary of a simplicial polytope, Stanley's proof of this case of the algebraic $g$-conjecture \cite{Stanleyg} can be adapted to give a simple proof of Theorem~\ref{thm:strongg}; see Remark~\ref{rem:polytopal}.

\medskip

Recall that a pure simplicial complex $\Delta$ of dimension  $d  - 1$ is a \emph{pseudomanifold} if  every $(d-2)$-dimensional face lies in exactly two facets, and each connected component of the geometric realization of $\Delta$ remains connected after we remove its $(d-3)$-skeleton. Any $k$-homology manifold is a pseudomanifold, and the constructions of $H(\Delta)$, $\overline{H}(\Delta)$, $\deg$, $H_F(\Delta)$, $\overline{H}_F(\Delta)$, and $\deg_F$ above are valid for pseudomanifolds. For connected oriented pseudomanifolds, we establish Theorem~\ref{thm:oddmultHodgeRiemann} and Theorem~\ref{thm:strongg} when $q = 0$ for all fields $k$ (see Theorem~\ref{thm:i=0} and Remark~\ref{rem:stronggq=0} respectively), and, if we further assume that $\Char k = 2$, then we establish 
Theorem~\ref{thm:oddmultHodgeRiemann} when $q = 1$ (see Remark~\ref{rem:q=1}) and Theorem~\ref{thm:strongg} in all degrees (see Theorem~\ref{thm:stronggpseudo}). 
We conjecture that all the results above hold for connected oriented pseudomanifolds over any field $k$ (Conjecture~\ref{conj:pseudoextension}). In Section~\ref{sec:discussion}, we discuss the sole obstruction to using our methods to prove this conjecture when $\Char k = 2$.

Our paper is organized as follows. In Section~\ref{sec:2}, we recall the construction and properties of the degree map. In Section~\ref{sec:special}, we compute some special cases which will be used in the proofs of the main theorems. In Section~\ref{sec:degreezero}, we prove the $q=0$ case of Theorems ~\ref{thm:oddmultHodgeRiemann} and ~\ref{thm:strongg}. In Sections~\ref{sec:almostanisotropy} and ~\ref{sec:char2}, we prove Theorem~\ref{thm:strongg} when $\Char k = 2$. 
In Section~\ref{sec:proofs}, we prove the main theorems. In Section~\ref{sec:discussion}, we give some examples and discuss possible extensions.

Throughout, we fix a connected oriented 
simplicial pseudomanifold  $\Delta$ over $k$ of dimension $d-1$ with vertex set $V$. We will sometimes further assume that $\Delta$ is a $k$-homology manifold, and we will 
occasionally
 omit the connected, oriented, and simplicial hypotheses. 
If $G$ is a face of $\Delta$ with vertices $\{ j_1, \dotsc,j_r\}$, we write $x_G \coloneqq x_{j_1} \dotsb x_{j_r}$ for the corresponding monomial in $K[\Delta]$. We will sometimes abuse notation and use $x_G$ to denote its image in $H(\Delta)$ or $\overline{H}(\Delta)$. See
\cite{StanleyCombinatoricsCommutative} 
for any undefined terminology. 

We will assume throughout that $d > 1$. If $d = 1$, 
the (not connected) case of a simplicial sphere of dimension $0$, i.e., $\Delta$ consists of two points, is discussed in Example~\ref{ex:d=1}.

\subsection*{Acknowledgements} 
We thank Ed Swartz for suggesting Example~\ref{ex:hilbertdependence}, and we thank the referees for their helpful comments. 
The work of the second author is partially supported by NSF grant DMS-2246399. The first and third authors thank the Institute for Advanced Study, where part of this work was conducted, for a pleasant environment.

\section{Degree maps}\label{sec:2}

We now discuss degree maps on artinian reductions of Stanley--Reisner rings of connected oriented simplicial
pseudomanifolds over $k$. Throughout the paper, we will compare degree maps associated to different artinian reductions.
The normalization of the degree map will be crucial in what follows, as the 
results of the introduction
can fail if we use an arbitrary isomorphism from $H^d(\Delta)$ to $K$. Explicitly, two such isomorphisms vary by multiplication by a nonzero element $\omega \in K$, and if $p = \dim \overline{H}^{q}(\Delta)$ is odd, then the determinant of the Hodge--Riemann form on $\overline{H}^{q}(\Delta)$ 
will vary by multiplication by $\omega^p = \omega \in K^{\times}/(K^{\times})^2$. 

\smallskip

We first discuss orientations over $k$ in the case when the characteristic of $k$ is not $2$. 
If $d > 1$, then an \emph{orientation} on a $(d-1)$-dimensional simplex is a choice of ordering of the vertices, up to changing the ordering by an even permutation. If $d = 1$, then an orientation on a $(d-1)$-dimensional simplex is a choice of $\epsilon \in \{1, -1\}$. 
An orientation on a $(d-1)$-dimensional simplex induces an orientation on each facet. If $d > 2$ and the simplex is ordered by $\{v_1 < \dotsb < v_d\}$, then we orient $\{v_2, \dotsc, v_d\}$ using the ordering $v_2 < \dotsb < v_d$, and we orient the facet which omits $v_i$ by changing the ordering by even permutations so that $v_i$ is first. If $d = 2$ and the simplex is $\{v_1 < v_2\}$, then we orient $\{v_1\}$ by $-1$ and orient $\{v_2\}$ by $1$. 

An \emph{orientation} of $\Delta$ over $k$ is a choice of orientation for each facet of $\Delta$ such that the two orientations on any $(d-2)$-dimensional simplex of $\Delta$ induced by the two facets containing it are opposite. In what follows, we fix a choice of orientation. 

If $k$ has characteristic $2$, then we say that any 
pseudomanifold
is oriented over $k$ by definition. 

For each facet $F = \{j_1 < \dotsb < j_d\}$, the orientation on $\Delta$ defines a sign $\epsilon_F \in \{1, -1\}$, which is $1$ if the permutation which takes $(j_1, \dotsc, j_d)$ to the ordering given by the orientation is even, and is $-1$ if this permutation is odd. If the characteristic of $k$ is $2$, then $\epsilon_F = 1$ by definition.

\smallskip

There is an explicit isomorphism $\deg \colon H^d(\Delta) \to K$, called the \emph{degree map}. This isomorphism was constructed by Brion \cite{BrionStructurePolytopeAlgebra}; see also \cite[Lemma 2.2]{KaruXiaoAnisotropy}.  
Recall that, for a subset $F = \{j_1 < \dotsb < j_{d}\}$ of $V$ of size $d$, $[F]$ is the determinant of the matrix whose $(i, m)$th entry is $a_{i,j_m}$.

\begin{proposition}\label{prop:degmap}
There is an isomorphism $\deg \colon H^d(\Delta) \to K$ of $K$-vector spaces such that, for any facet $F$ of $\Delta$,
we have 
\begin{equation}\label{eq:degreefacet}
	\deg(x_{F}) = \frac{\epsilon_F}{[F]}.
\end{equation}
\end{proposition}

In particular, if $k$ does not have characteristic $2$, then the degree map associated to the opposite orientation is the negative of the original degree map. 

More generally, consider $d$ elements $\mu = (\mu_1, \dotsc, \mu_d)$  in $K[\Delta]$ of degree $1$, with $\mu_i = \sum_{j \in V} \mu_{i,j} x_j$ for some  $\mu_{i,j} \in K$.
Let $k[a_{i,j}]$ denote the polynomial ring 
$k[a_{i,j}]_{1 \le i \le d, \, 1 \le j \le n}$ with fraction field $K$, and  consider the $k$-algebra homomorphism $\ev_\mu \colon k[a_{i,j}] \to K$ defined by
\[
\ev_\mu(a_{i,j}) = \mu_{i,j}. 
\]
We will use the following criterion
for the elements of $\mu$ to be a linear system of parameters (l.s.o.p.).
\begin{proposition}
\cite[Lemma III.2.4]{StanleyCombinatoricsCommutative}	
	\label{prop:stanleycriterion}
Consider $d$ elements $\mu = (\mu_1, \dotsc, \mu_d)$  in $K[\Delta]$ of degree $1$.
 Then $\mu_1, \dotsc, \mu_d$ is an l.s.o.p. if and only if $\ev_\mu([F]) \neq 0$ for each facet $F$ of $\Delta$. 
\end{proposition}

Suppose that $\mu = ( \mu_1, \dotsc, \mu_d )$ is an l.s.o.p. Let $H_{\mu}(\Delta) \coloneqq K[\Delta]/(\mu_1, \dotsc, \mu_d)$. 
We still have $\dim H^d_{\mu}(\Delta) = 1$ (see, for example,
\cite[Corollary~3.2]{TWHomological}), 
and so 
the degree map described in Proposition~\ref{prop:degmap} ``specializes'' to an isomorphism $\deg_\mu \colon H^d_{\mu}(\Delta) \to K$ of $K$-vector spaces  such that, for a fixed choice of facet $F$ of $\Delta$,
\begin{equation}\label{eq:degreefacetmu}
	\deg_\mu(x_F) = \frac{\epsilon_F}{\ev_\mu([F])}.
\end{equation} 
We will verify below that \eqref{eq:degreefacetmu} is independent of  the choice of facet $F$. 
We also have a well-defined Gorenstein quotient $\overline{H}_\mu(\Delta)$, i.e., the quotient of $H_\mu(\Delta)$ by the ideal $(y \in H_\mu(\Delta) : (y  z)_d = 0 \text{ for all }z \in H_\mu(\Delta))$, where $w_d$ denotes the degree $d$ component of $w$ in $H_\mu(\Delta)$. 
For example, as in the statement of 
Theorem~\ref{thm:strongg},
let $F$ be a subset of $V$ of size $d$ which is not a facet of $\Delta$, and set $\theta_1^F = \sum_{j \not \in F} a_{1, j} x_j$. Then $\theta_F = (\theta_1^F, \theta_2, \dotsc, \theta_d)$ is an l.s.o.p., and we write 
$H_F(\Delta) \coloneqq H_{\theta_F}(\Delta)$, $\overline{H}_F(\Delta) \coloneqq \overline{H}_{\theta_F}(\Delta)$, and $\deg_F \coloneqq \deg_{\theta_F}$.

\smallskip

We now describe two known techniques that can be used to compute the degree map. We first recall the following application of Cramer's rule; see, e.g., \cite[Proposition~2.1]{PapadakisPetrotouGenericAnisotropy1Spheres}. Below, $\sgn(\pi) \in \{ \pm 1 \}$ denotes the sign of a permutation $\pi$. 

\begin{lemma}\label{lem:Gaussianelim}
	Let $\mu = ( \mu_1, \dotsc, \mu_d )$ be an l.s.o.p.	
	Let $F = \{j_1 < \dotsb < j_{d}\}$ be a subset of $V$ of size $d$.
	Fix $1 \le m \le d$. Then 
\begin{equation}\label{eq:gaussianelim}
		\ev_\mu([F]) x_{j_m} = - \sum_{v \in V \smallsetminus F}  \sgn(\pi_v)	
	\ev_\mu([F \cup \{ v \} \smallsetminus \{ j_m\}]) x_v  \in H^1_\mu(\Delta),
\end{equation}
	where $\pi_v \in S_d$ is the permutation such that 
	the elements of $\pi_v   (j_1,\ldots,j_{m - 1},v,j_{m + 1},\ldots,j_d)$ are in increasing order.

\end{lemma}

Suppose that $F$ and $F'$ are facets of $\Delta$. 
Since $\Delta$ is a connected pseudomanifold,
there is a sequence of facets $F = F_1, F_2, \ldots,F_s = F'$, where $F_j$ and $F_{j + 1}$ meet along a common face of dimension $d-2$ for $1 \le j < s$. Suppose that $F = \{ j_1 < \cdots < j_d \}$ and $F'$ meet along the common face $F \smallsetminus \{ j_m \}$. Since $F$ and $F'$ are the only facets containing $F\smallsetminus \{j_m\}$, multiplying \eqref{eq:gaussianelim} by $x_F/x_{j_m}$ and tracing through the signs yields that $\epsilon_{F'} \operatorname{ev}_{\mu}([F']) x_{F'} = \epsilon_{F} \operatorname{ev}_{\mu}([F])x_{F} \in H^d_\mu(\Delta)$. We conclude that \eqref{eq:degreefacetmu} holds for any facet $F$ of $\Delta$. 

Given a nonzero monomial $x_{j_1}^{b_{j_1}} \dotsb x_{j_s}^{b_{j_s}} \in K[\Delta]$ with each $b_{j} > 0$, define its \emph{support} to be the face $\{j_1, \dotsc, j_s\}$ of $\Delta$. Suppose that the above monomial is not squarefree, i.e., $b_{j_m} > 1$ for some $1 \le m \le s$. Let $F$ be a facet containing the support $\{j_1, \dotsc, j_s\}$. Then Lemma~\ref{lem:Gaussianelim} implies that 
\begin{equation}\label{eq:makesquarefree} 
	x_{j_1}^{b_{j_1}} \dotsb x_{j_s}^{b_{j_s}} = - \frac{1}{\ev_\mu([F])}\sum_{v \in V \smallsetminus F}  \sgn(\pi_v)
	\ev_\mu([F \cup \{ v \} \smallsetminus \{ j_m\}]) \frac{x_v  x_{j_1}^{b_{j_1}} \dotsb x_{j_s}^{b_{j_s}}}{x_{j_m}} \in H_\mu(\Delta),
\end{equation}
for some permutations $\pi_v$ as defined in Lemma~\ref{lem:Gaussianelim}. Importantly, all nonzero monomials on the right-hand side of \eqref{eq:makesquarefree} have support strictly containing the support of $x_{j_1}^{b_{j_1}} \dotsb x_{j_s}^{b_{j_s}}$. Hence we may compute the degree of any monomial by using \eqref{eq:makesquarefree} to repeatedly increase the size of the support.

Throughout the paper, we will repeatedly use the following lemma to compare different artinian reductions.
Let $R \subset K$ be the localization of $k[a_{i,j}]$ at the irreducible polynomials  $\{ [F] : F \textrm{ facet of } \Delta \}$. By Proposition~\ref{prop:stanleycriterion}, $\ev_\mu$  extends to a $k$-algebra homomorphism $\ev_\mu \colon R \to K$. 

\begin{lemma}\label{lem:degmu}
		Let $\mu = ( \mu_1, \dotsc, \mu_d )$ be an l.s.o.p.
	Let $g \in k[x_1,\ldots,x_n]_d$ be a polynomial of degree $d$. 
		Then $\deg(g) \in R$ and $\deg_\mu(g) = \ev_\mu(\deg(g)) $. 
	
\end{lemma}
\begin{proof}
	It is enough to consider the case when $g$ is a monomial. If $g$ is squarefree, then the result follows from \eqref{eq:degreefacet} and \eqref{eq:degreefacetmu}. If $g$ is not squarefree, then 
	the result follows by using \eqref{eq:makesquarefree} to repeatedly increase the size of the support.
\end{proof}

We will apply Lemma~\ref{lem:degmu} in Sections \ref{sec:special}, \ref{sec:degreezero}, and \ref{sec:proofs} in combination with the following simple observation. We will often use the remark below with $P = [F]$ for some non-face  $F$ of size $d$.

 \begin{remark}\label{r:ordernonfaces}
		
		Consider an element $f \in R$. 
		Let $P \in k[a_{i,j}]$ be an irreducible polynomial, and suppose that there is an l.s.o.p.~$\mu$ with $\ev_{\mu}(P) = 0$, but $\ev_\mu(f) \not=0$. We claim that $\ord_{P}(f) = 0$. 
		Indeed, because $\ev_{\mu}(P) = 0$,  $P$ is not a scalar multiple of $[F]$ for any facet $F$ of $\Delta$, so $\ord_P(f) \ge 0$.  But $P$ cannot divide $f$ to positive order as $\ev_{\mu}(f) \not=0$. 
\end{remark}

\smallskip

We 
next
recall a formula for the degree map 
due to Karu and Xiao. It is closely related to the work of Brion \cite{BrionStructurePolytopeAlgebra} 
and that of Lee \cite{Lee}. 
To state the formula, we 
define $\hat{V} \coloneqq \{ 0 \} \cup V$ and $\hat{K} \coloneqq K(a_{i,0}: 1 \le i \le d)$. For a subset $\hat{F} = \{j_1 < \dotsb < j_{d}\}$ of $\hat{V}$ of size $d$, let $[\hat{F}]$ be the determinant of the $d \times d$ matrix whose $(i, m)$th entry is $a_{i,j_m}$.

\begin{proposition}\label{prop:karuxiaodegree}
		\cite[Lemma~3.1, Theorem~3.2]{KaruXiaoAnisotropy}
Let $g \in K[x_1,\ldots,x_n]_d$ be a polynomial of degree $d$. For any facet $F = \{j_1 < \dotsb < j_{d}\}$ of $\Delta$, let $g_F(t_1,\ldots,t_d)$ be obtained from $g$ by setting $x_i$ to zero for $i \notin F$ and setting $x_{j_m} = t_m$ for $1 \le m \le d$. Let $X_{F,m} \coloneqq (-1)^m[F \cup \{ 0 \} \smallsetminus \{ j_m\} ] \in \hat{K}$ for $1 \le m \le d$. 
Then 
\begin{equation}\label{eq:KaruXiaoformula}
	\deg(g) = 
	\sum_{F \textrm{ facet of } \Delta} \frac{\epsilon_F g_F(X_{F,1},\ldots,X_{F,d})}{[F] \prod_{m = 1}^d X_{F,m}}.
\end{equation}
\end{proposition}

In particular, the expression in \eqref{eq:KaruXiaoformula} lies in $K$. 
When $g \in k[x_1,\ldots,x_n]$, 
this formula specializes to a formula for all linear systems of parameters.
 Explicitly, 
suppose that $\mu = ( \mu_1, \dotsc, \mu_d )$ is an l.s.o.p. 
Recall that we have an 
 evaluation map $\ev_\mu \colon k[a_{i,j}] \to K$ 
defined by $\ev_\mu(a_{i,j}) = \mu_{i,j}$ for $1 \le i \le d$ and $1 \le j \le n$. This naturally extends to 
 a $k$-algebra homomorphism
$\hat{\ev}_\mu \colon k[a_{i,j}]_{1 \le i \le d, 0 \le j \le n} \to \hat{K}$ such that $\hat{\ev}_\mu(a_{i,0}) = a_{i,0}$ for $1 \le i \le d$. 
For the statement below, observe that if  $F  = \{j_1 < \dotsb < j_{d}\}$ is a facet of $\Delta$ and $1 \le m \le d$, then $\hat{\ev}_\mu([F \cup \{ 0 \} \smallsetminus \{ j_m \} ])$ is nonzero since it specializes (up to a sign) to $\ev_\mu([F])$ by setting $a_{i,0}$ to $\mu_{i,j_m}$ for $1 \le i \le d$, and by Proposition~\ref{prop:stanleycriterion}, $\ev_{\mu}([F]) \not=0$.

\begin{corollary}\label{cor:KaruXiaoformu}
	Suppose that $\mu = ( \mu_1, \dotsc, \mu_d )$ is an l.s.o.p.
	Let $g \in k[x_1,\ldots,x_n]_d$ be a polynomial of degree $d$.  For any facet $F = \{j_1 < \dotsb < j_{d}\}$ of $\Delta$, let $g_F(t_1,\ldots,t_d)$ be obtained from $g$ by setting $x_i$ to zero for $i \notin F$ and setting $x_{j_m} = t_m$ for $1 \le m \le d$. Let $X_{F,\mu,m} \coloneqq (-1)^m \hat{\ev}_\mu([F \cup \{ 0 \} \smallsetminus \{ j_m\} ]) \in \hat{K}$ for $1 \le m \le d$. Then
	
	\[
		\deg_\mu(g) = 
	\sum_{F \textrm{ facet of } \Delta} \frac{\epsilon_F g_F(X_{F,\mu,1},\ldots,X_{F,\mu,d})}{\ev_\mu([F]) \prod_{m = 1}^d X_{F,\mu,m}}.
	\]
\end{corollary}
\begin{proof}
	Recall that $R \subset K$ is the localization of $k[a_{i,j}]$ at the irreducible polynomials  $\{ [F] : F \textrm{ is a facet of } \Delta \}$, and $\ev_\mu$ maps $R$ to $K$ which is contained in $\hat{K}$. With the notation of Proposition~\ref{prop:karuxiaodegree}, let $\hat{R} \subset \hat{K}$ be the localization of $k[a_{i,j}]_{1 \le i \le d, 0 \le j \le n}$ at the irreducible polynomials $\{ X_{F,m} : F \textrm{ facet of } \Delta, \, 0 \le m \le d \}$, where $X_{F,0} \coloneqq [F]$. 
	Then $\hat{\ev}_\mu$ extends to a $k$-algebra homomorphism 
	$\hat{\ev}_\mu \colon \hat{R} \to \hat{K}$ such that  $\ev_\mu$ is the restriction of 
	$\hat{\ev}_\mu$ to $R \subset \hat{R}$.
	By Lemma~\ref{lem:degmu}, both sides of \eqref{eq:KaruXiaoformula} lie in $\hat{R}$ and $\hat{\ev}_\mu(\deg(g)) = \deg_\mu(g)$. The result now follows by applying $\hat{\ev}_\mu$ to both sides of 
	\eqref{eq:KaruXiaoformula}. 
\end{proof}

\smallskip

We next
 show that the degree can be computed ``locally'' on $\Delta$, in an appropriate sense. This will be used in Section~\ref{sec:proofs} to reduce to the special cases computed in Section~\ref{sec:special}.
The \emph{closed star} $\Star_\Delta(G)$ of a face $G$ of $\Delta$ is the simplicial complex consisting of 
all faces $G'$ of $\Delta$ that contain $G$, together with their subfaces. 
Let $G = \{j_1, \dotsc, j_s\}$ 
be a face of $\Delta$, and let $S$ be the set of vertices in $\Star_\Delta(G)$.  Let $K_S$ be the subfield of $K$ generated over $k$ by $a_{i,j}$, where $1 \le i \le d$ and $j \in S$.  Let $\Delta'$ be another 
 connected oriented simplicial 
pseudomanifold  over $k$
 of dimension $d-1$, with vertex set $V'$ and a face 
$G' = \{j_1', \dotsc, j_s'\}$. 
Let $K' = k(a'_{i,j})_{1 \le i \le d, j \in V'}$ be the field of coefficients for $\overline{H}(\Delta')$. 
Suppose that there is an orientation-preserving isomorphism of simplicial complexes $\tau \colon \Star_\Delta(G) \to \Star_{\Delta'}(G')$ that maps $j_m$ to $j_m'$ for $1 \le m \le s$.
Then $\tau$ allows us to identify $K_S$ with a subfield of $K'$.  
Let $\deg_\Delta \colon H^d(\Delta) \to K$ and $\deg_{\Delta'} \colon H^d(\Delta') \to K'$ denote the degree maps for $\Delta$ and $\Delta'$ respectively.

\begin{lemma}\label{lem:deglocal}
	With the notation above, let $x_{j_1}^{b_1} \dotsb x_{j_s}^{b_s}$ be a monomial of degree $d$ with support $G$. Then  $\deg_\Delta(x_{j_1}^{b_1} \dotsb x_{j_s}^{b_s}) \in K_S$.
Using the identification of $K_S$ with a subfield of $K'$, we have
$$
\deg_\Delta(x_{j_1}^{b_1} \dotsb x_{j_s}^{b_s}) =
\deg_{\Delta'}(x_{j_1'}^{b_1} \dotsb x_{j_s'}^{b_s}).
$$
\end{lemma}

\begin{proof}
The only nonzero terms in the right-hand side of the formula for $\deg_{\Delta}(x_{j_1}^{b_1} \dotsb x_{j_s}^{b_s})$ in
\eqref{eq:KaruXiaoformula}
 are those corresponding to facets in $\Star_\Delta(G)$, and those terms lie in $K_S(a_{i,0})_{1 \le i \le d}$ and are equal (up to a global sign) to the corresponding terms in the formula for $\deg_{\Delta'}(x_{j'_1}^{b_1} \dotsb x_{j'_s}^{b_s})$.
\end{proof}

For the remainder of the section, we assume that $k$ has characteristic $2$. We recall some differential operator identities that will be used in the proof of Theorem~\ref{thm:strongg}.
For $1 \le i \le d$ and $1 \le j \le n$, 
let $\partial_{a_{i,j}} \colon K \to K$ be the partial derivative with respect to $a_{i,j}$. 
Since $k$ has characteristic $2$, 
\begin{equation}\label{eq:derivativesquareszero}
	\partial_{a_{i,j}} \lambda^2 = 0
\end{equation}  for all $\lambda \in K.$ 
For a sequence $I = (i_1, \dotsc, i_r)$ of elements of $V$, let $\partial_I = \partial_{a_{1, i_1}} \dotsb \partial_{a_{d, i_d}}$ if $r = d$; note that this depends on the ordering of $I$. Let $x_I = x_{i_1}\cdots x_{i_r}$. 
If we write $x_I = x_1^{b_1}\cdots x_{n}^{b_n}$ for some nonnegative integers $b_i$, then we define $\sqrt{x_I}$ to be $x_1^{b_1/2}\cdots x_{n}^{b_n/2}$ if each $b_i$ is even, and otherwise we define $\sqrt{x_I}$ to be $0$.  
The following result of Karu and Xiao was conjectured in 
\cite[Conjecture 14.1]{PapadakisPetrotoug}. 

\begin{proposition}\label{thm:karuxiao}\cite[Theorem 4.1]{KaruXiaoAnisotropy}
	Let $\Delta$ be a 
	pseudomanifold of dimension $d-1$, and assume that $\Char k = 2$. 
	Let $I$ and $J$ be sequences of elements of $V$ of size $d$. 
	Then
	\begin{equation*}
		\partial_I \deg (x_J) = \deg(\sqrt{x_I x_J})^2.
	\end{equation*}
\end{proposition}

Recall that for a non-face $F$ of $\Delta$ of size $d$, 
 $\theta_1^F = \sum_{j \not \in F} a_{1, j} x_j$,  $\theta_F = (\theta_1^F, \theta_2, \dotsc, \theta_d)$,
$\overline{H}_F(\Delta) = \overline{H}_{\theta_F}(\Delta)$, and 
$\deg_F = \deg_{\theta_F}$.  
We have the following corollary. 

\begin{corollary}\label{cor:karuxiaoanalogue}
	Let $\Delta$ be a  
	pseudomanifold of dimension $d-1$, and assume that $\Char k = 2$. Let $F$ be a non-face of $\Delta$ of size $d$, and let $I$ and $J$ be sequences of elements of $V$ of size $d$. Assume that the first element of $I$ does not lie in $F$. Then
	\begin{equation*}
		\partial_I \deg_F (x_J) = \deg_F(\sqrt{x_I x_J})^2.
	\end{equation*}
\end{corollary} 
\begin{proof}
	Recall that $R \subset K$ is the localization of $k[a_{i,j}]$ at 
	$\{ [G] : G \textrm{ facet of } \Delta \}$, and that
	$\ev_{\theta_F}$ is the evaluation map from $R$ to $K$ associated to $\theta_F$.
	By Lemma~\ref{lem:degmu}, $\deg(x_J) \in R$ and $\deg_F(x_J) = \ev_{\theta_F}(\deg(x_J))$. Observe that $\partial_{I}$ restricts to a map from $R$ to itself. Since the first element of $I$ does not lie in $F$, 
	$$\partial_{I}\ev_{\theta_F}(\lambda) = \ev_{\theta_F}(\partial_{I} \lambda)$$ 
	for 
	any $\lambda \in R$.
	Using Proposition~\ref{thm:karuxiao} and Lemma~\ref{lem:degmu}, 
	we compute 
	\[
	\partial_I \deg_F (x_J) = \partial_I \ev_{\theta_F}(\deg(x_J)) = \ev_{\theta_F}(\partial_I \deg(x_J)) = \ev_{\theta_F}(\deg(\sqrt{x_I x_J})^2) = \deg_F(\sqrt{x_I x_J})^2. \qedhere
	\]
\end{proof}

The following analogue of \cite[Corollary~4.2]{KaruXiaoAnisotropy} will be useful in what follows. Although the proof is identical to that of \cite[Corollary~4.2]{KaruXiaoAnisotropy}, we recall it for the benefit of the reader.

\begin{corollary}\label{cor:karuxiaoanalogue2}
	Let $\Delta$ be a  
	pseudomanifold 
	of dimension $d-1$, and assume that $\Char k = 2$. Let $F$ be a non-face of $\Delta$ of size $d$, let $h \in K[x_1,\ldots,x_n]_q$ for some $0 \le q \le d/2$, and let $I$ and $J$ be sequences of elements of $V$ of size $d$ and $d - 2q$ respectively. Assume that the first element of $I$ does not lie in $F$. Then 
	\begin{equation*}
		\partial_I \deg_F (h^2 x_J) = \deg_F(h \sqrt{x_I x_J})^2.
	\end{equation*}
\end{corollary}
\begin{proof}
	Write $h = \sum_L \lambda_L x_L$ 
	for some sequences $L$ of elements of $V$ of size $q$ and some $\lambda_L \in K$. By \eqref{eq:derivativesquareszero},
	$\partial_I$ commutes with multiplication by elements of $K^2$. Using  Corollary~\ref{cor:karuxiaoanalogue}, we compute
	\begin{align*}
		\partial_I \deg_F (h^2 x_J) &= \sum_L \partial_I (\lambda_L^2 \deg_F (x_L^2 x_J)) \\
		&= \sum_L \lambda_L^2 \partial_I \deg_F (x_L^2 x_J) \\
		&= \sum_L \lambda_L^2 \deg_F(x_L\sqrt{x_I x_J})^2 \\
		&=  \deg_F(h \sqrt{x_I x_J})^2. \qedhere
	\end{align*}
\end{proof}

For another recent adaptation of 
\cite[Theorem 4.1]{KaruXiaoAnisotropy}
to a nongeneric situation, we refer the reader to \cite[Lemma 4.2 and Corollary~4.3]{Oba}.

\section{Some important special cases}\label{sec:special}

In this section, we analyze several important special cases. In order to prove the main theorems, we will need a detailed understanding of
the boundary of the $d$-dimensional simplex, i.e., the complex $S^{d-1}$ with vertex set $\{1, \dotsc, d + 1\}$ and unique minimal non-face $\{1, \dotsc, d + 1\}$.
We continue to assume that $d > 1$ unless otherwise stated. Let $k$ be any field.

Recall that the polynomials $\{ [G] : G \subset V, \, |G| = d \}$ in $k[a_{i,j}]$ are irreducible (see, for example, Lemma~\ref{lemma:irreducible}).  We will need the following lemma to analyze these special cases.

\begin{lemma}\label{lem:ordAiszero}
	
	If $A$ can be written as a $k$-linear combination  of
	 $\{ [G] : G \subset V, \, |G| = d \}$ where at least two coefficients are nonzero, then $\ord_{[G]}(A) = 0$ for all $G \subset V$ of size $d$. 
\end{lemma}
\begin{proof}
	The
	$k$-algebra generated by the irreducible polynomials $\{ [G] : G \subset V, \, |G| = d \}$ in $k[a_{i,j}]$
	is isomorphic to the Pl\"ucker ring, i.e., the homogeneous coordinate ring of the Grassmannian of $d$-planes in $k^{n}$. In particular, since the Pl\"ucker relations all have degree strictly greater than $1$, 
	 the polynomials $\{ [G] : G \subset V, \, |G| = d \}$ are linearly independent over $k$.
	 Each $[G]$ is homogeneous of degree $d$. Hence $A$ is also homogeneous of degree $d$. 
If $\ord_{[G']}(A) > 0$, then, by comparing degrees, $A = \lambda [G']$ for some $\lambda \in k$, contradicting the assumption that at least two coefficients are nonzero. 
\end{proof}

We can now analyze the boundary of the $d$-dimensional simplex for $d > 1$. 

\begin{example}\label{ex:boundarysimplex}

	Let $S^{d-1}$ be the boundary of the $d$-dimensional simplex  with vertex set $V = \{1, \dotsc, d + 1\}$.	By Lemma~\ref{lem:Gaussianelim}, for $1 \le m,p \le d + 1$, we have
	\begin{equation}\label{eq:Gaussianelimboundarysphere}
					[V \smallsetminus \{ p \}] x_{m} = (-1)^{|p - m|} 
						[V \smallsetminus \{ m \}] x_p  \in H^1(S^{d-1}).
	\end{equation}	
	Fix $0 \le q \le d/2$. 
	A basis for $H^q(S^{d-1}) = \overline{H}^q(S^{d-1})$ is $x_1^q$.  Using \eqref{eq:Gaussianelimboundarysphere},	we compute 
	\begin{align*}
	 	 \left( \prod_{m = 2}^d (-1)^{m - 1} [V \smallsetminus \{ m \}] \right) \deg(x_1^d) &= [V \smallsetminus \{ 1 \}]^{d - 1} \deg(x_1  \cdots x_d), 
	\end{align*}
	and hence, for some $\epsilon \in \{ \pm 1\}$, we have
	\[
	\deg(x_1^d) =  \frac{\epsilon [V \smallsetminus \{ 1 \}]^{d}}{ \prod_{m = 1}^{d + 1}  [V \smallsetminus \{ m \}] }.
	\]
	Using \eqref{eq:Gaussianelimboundarysphere}, we compute
	\begin{align*}
		[V \smallsetminus \{ 1 \}]^{d - 2q} \deg(\ell^{d - 2q}   x_1^{2q}) = 
		A^{d - 2q} \deg(x_1^d),
	\end{align*} 
			where $A \coloneqq \sum_{m = 1}^{d + 1} (-1)^{m - 1}[V \smallsetminus \{ m \}]$. 
			By Lemma~\ref{lem:ordAiszero},  $\ord_{[G]}(A) = 0$ for any subset $G$ of $V$ of size $d$.
	Putting this together gives
	\begin{equation}\label{eq:degx1dsphere}
			\deg(\ell^{d - 2q}  x_1^{2q}) =  \frac{\epsilon A^{d - 2q} [V \smallsetminus \{ 1 \}]^{2q}}{ \prod_{m = 1}^{d + 1}  [V \smallsetminus \{ m \}] }.
	\end{equation}
	Let $D_{q}$ be the image of $\deg(\ell^{d - 2q}  x_1^{2q})$ in $K^{\times}/(K^{\times})^2$. By Lemma~\ref{lem:ordAiszero}, $\ord_{[G]}(A) = 0$ for all facets $G$ of $S^{d-1}$, so
	\[
	D_{q} = 		\begin{cases}
		\epsilon \prod_{G \text{ facet}} [G] &\textrm{ if } d \textrm{ is even} \\
		\epsilon A	\prod_{G \text{ facet}} [G] &\textrm{ if } d \textrm{ is odd}. \\
	\end{cases}
	\]
	We conclude that 
	Theorem~\ref{thm:oddmultHodgeRiemann}
	holds in this case (without any assumptions on $k$). Observe that 
	Theorem~\ref{thm:strongg}
holds vacuously since $S^{d-1}$ has no non-faces of size $d$. 	
	
\end{example}

We will also need to analyze the case of $S^0$, i.e., the disjoint union of two vertices. 
Although $S^0$ is not connected, $H(S^0)$ is a Gorenstein ring with a well-defined degree map, and we verify directly that the conclusion of Theorem~\ref{thm:oddmultHodgeRiemann}
holds in this case (without any assumptions on $k$).

\begin{example}\label{ex:d=1}
	Let $d = 1$, $V = \{ 1,2\}$, and consider the complex $S^0$ with vertex set $\{1, 2\}$ and minimal non-face $\{1, 2\}$.  We orient $S^0$ by assigning $-1$ to the facet $\{1\}$ and assigning $1$ to the facet $\{2\}$. We have
	\[
	\deg(\ell) = \deg(x_1) + \deg(x_2) = -\frac{1}{a_{1,1}} + \frac{1}{a_{1,2}} = \frac{a_{1,1} - a_{1,2}}{a_{1,1}a_{1,2}}. 
	\]
\end{example}

The next lemma will be crucial to the proof of 
Theorem~\ref{thm:oddmultHodgeRiemann} when $q = 0$. 
Recall that $d > 1$ and let $\Sigma$ be the suspension of the boundary of the $(d-1)$-dimensional simplex, i.e., the complex with $V=\{1,\ldots,d+2\}$ and minimal non-faces $\{1,\ldots,d\}$ and $\{d+1,d+2\}$; in particular,  $\ell=x_1+\cdots+x_{d+2}$.

\begin{lemma}\label{lem:HL0sigma}
For every non-face $G$ of $\Sigma$ of size $d$, we have
$\ord_{[G]}(\deg(\ell^d)) = 0$.
\end{lemma}
\begin{proof}
By Lemma~\ref{lem:degmu} and Remark~\ref{r:ordernonfaces}, 
it is enough to show
that there is an l.s.o.p. $\mu$ which has $\ev_{\mu}([G]) = 0$, but $\deg_{\mu}(\ell^d) = \ev_\mu(\deg(\ell^d)) \not=0$. 

We consider two cases. First assume that $\Char k$ does not divide $d$. 
Set $\mu_i = a_{i,1}x_1 + \dotsb + a_{i,d} x_d$ for $1 \le i < d$, and set $\mu_d = a_{d,d+1}x_{d+1} + a_{d,d+2} x_{d+2}$. Because $\Sigma = S^{d-2} * S^{0}$, we see that $H_{\mu}(\Sigma) = H(S^{d-2}) \otimes H(S^0)$.
Furthermore, we can write $\ell = \ell_1 + \ell_2$, where $\ell_1 = x_1 + \dotsb + x_d$ and $\ell_2 = x_{d+1} + x_{d+2}$. We have 
$\ell_1^d = 0$
and $\ell_2^2 = 0$, and we see that
$$\deg_{\mu}(\ell^d) = d \deg_{\mu}(\ell_1^{d-1}  \ell_2) = d \deg_{S^{d-2}}(\ell_1^{d-1}) \deg_{S^0}(\ell_2).$$
Since $\deg_{S^{d-2}}(\ell_1^{d-1}) \neq 0$ and  $\deg_{S^0}(\ell_2) \neq 0$ by  Example~\ref{ex:boundarysimplex} and Example~\ref{ex:d=1} respectively, we deduce that 
$\deg_{\mu}(\ell^d) \neq 0$. 
It remains to show that $\ev_{\mu}([G]) = 0$. 
Since $G$ is a non-face, either $G = \{1, \dotsc, d\}$ or $G$ contains $\{ d + 1, d + 2\}$. In the former case, $\ev_{\mu}([G])$ is the determinant of a matrix whose $d$th row is identically zero. In the latter case, $\ev_{\mu}([G])$ is the determinant of a matrix whose last two columns are identically zero except in the  $d$th row and hence are linearly dependent. 	
	
	Next assume that $\Char k$ divides $d$. 
	Since $G$ is a non-face, either $G = \{1, \dotsc, d\}$ or $G$ contains $\{ d + 1, d + 2\}$. If $d = 2$, then after relabeling if necessary, we may assume that $G = \{ 1, 2 \}$. If $d > 2$, then $G$ has size strictly greater than 2 and hence contains an element in $\{1, \dotsc, d\}$. In either case, we may assume without loss of generality that $d \in G$.	 
In this case, 
set $\mu_i = x_i - x_{i + 1}$ for $1 \le i \le d - 2$, set $\mu_{d - 1} = a_{d - 1,d - 1} x_{d - 1} - x_d + x_{d + 1}$, and set $\mu_d = x_{d + 1} - x_{d + 2}$.  We have the following description of $H_\mu(\Sigma) = \overline{H}_\mu(\Sigma)$:
\begin{align*}
	H_\mu(\Sigma) &=  K[x_1,\dotsc,x_{d + 2}]/(x_1\cdots x_d, x_{d + 1}x_{d + 2}, \mu_1, \dotsc, \mu_d) 
	\\ &\cong K[x_1,x_{d + 1}]/(x_1^{d - 1}(a_{d - 1,d - 1} x_1 +  x_{d + 1}), x_{d + 1}^2). 
\end{align*}
In particular, $H_\mu(\Sigma)$ is finite-dimensional, so $\mu$ is an l.s.o.p.  
The computations below all occur in $H_\mu(\Sigma)$. 
We have 
\[
\ell = x_1 + \cdots + x_{d + 2} = (d - 1 + a_{d - 1,d - 1}) x_1 + 3 x_{d + 1}. 
\]
Since $x_{d + 1}^2 = 0$ and $\Char k$ divides $d$,
\[
\ell^d = (d - 1 + a_{d - 1,d - 1})^d x_1^d + 3 d (d - 1 + a_{d - 1,d - 1})^{d  - 1} x_1^{d - 1} x_{d + 1} = (d - 1 + a_{d - 1,d - 1})^d x_1^d. 
\]
Hence $\ell^d$ is nonzero if and only if $x_1^d$ is nonzero. On the other hand, using the relation $x_1 x_2 \cdots x_d = 0$, we have
\begin{align*}
	x_1^d &= x_1 x_2 \cdots x_{d - 1} x_1 \\
	      &= x_1 x_2 \cdots x_{d - 1} a_{d - 1,d - 1}^{-1}(x_d -  x_{d + 1}) \\
	      &= - a_{d - 1,d - 1}^{-1}  x_1 x_2 \cdots x_{d - 1} x_{d + 1}. 
\end{align*}
The latter is nonzero since $x_1 x_2 \cdots x_{d - 1} x_{d + 1}$ is the image of the monomial corresponding to the facet $F = \{1,2,\ldots,d - 1,d + 1\}$ (see \eqref{eq:degreefacetmu}). It remains to show that $\ev_{\mu}([G]) = 0$. 
Since $G$ is a non-face, either $G = \{1, \dotsc, d\}$ or $G$ contains $\{ d + 1, d + 2\}$. In the former case, $\ev_{\mu}([G])$ is the determinant of a matrix whose $d$th row is identically zero. In the latter case, since we assumed that $d \in G$, we have $\{ d , d + 1, d + 2\}$ in $G$, and  $\ev_{\mu}([G])$ is the determinant of a matrix whose last three columns are identically zero except in the  $(d - 1)$st and $d$th row and hence are linearly dependent. 	
\end{proof}

\section{The degree zero case}\label{sec:degreezero}

In this section, we prove the following result that settles the $q=0$ case of Theorem~\ref{thm:oddmultHodgeRiemann}. This case works over a field of any characteristic. 

\begin{theorem}\label{thm:i=0}
Let $\Delta$ be a connected oriented pseudomanifold of dimension $d-1$ with vertex set $V$. Let $F$ be a subset of $V$ of size $d$. Then
$$\ord_{[F]}(\deg(\ell^d)) = \begin{cases} -1 & \text{if }F \text{ is a facet of }\Delta \\ 0 & \text{otherwise}.\end{cases}$$
\end{theorem}

 Recall that throughout we are assuming that $d > 1$. We first prove a lemma which will be used in the proof of Theorem~\ref{thm:i=0}. The case when $p = m$ is very well-known; see, for example, \cite[Theorem~61.1]{BocherIntroductionHigherAlgebra}.

\begin{lemma}\label{lemma:irreducible}
	For some $1 < p \le m$, let $N$ be the $m \times m$ matrix with $N_{i,j} = a_{i,j}$ if $i = 1$ and $j \le p$ or if $i > 1$, and $N_{i,j} = 0$ for $i = 1$ and $j > p$. Then $\det N$ is an irreducible polynomial in $k[a_{i,j}]$. 
\end{lemma}

\begin{proof}
	Suppose that $\det N = f  g$, where $f, g \in k[a_{i,j}]$. Because $\det N$ is linear in $a_{1,1}$, we see that $a_{1,1}$ must occur in exactly one of $f$ and $g$, say $f$. Because $p > 1$, the variables $a_{i, 1}$ appear in $\det N$ for $i \ge 1$. Those variables must also only occur in $f$, because $a_{1,1} a_{i,1}$ does not appear in $\det N$. 
	This implies that $a_{i, j}$ must also occur only in $f$ for each $j$, because $a_{i,1}a_{i,j}$ does not appear in $\det N$. 
	We conclude that $g$ is a unit. 
\end{proof}

In particular, Lemma~\ref{lemma:irreducible} implies that the polynomial $\det N$ defines a valuation on $K$. 
We now begin proving Theorem~\ref{thm:i=0}. We first deal with the case when $F$ is a facet.

\begin{proposition}\label{prop:h0facetcase}
	Let $F$ be a facet of $\Delta$. Then 
	$$\ord_{[F]}(\deg(\ell^d)) = -1.$$
\end{proposition}

\begin{proof}
	Using Proposition~\ref{prop:karuxiaodegree} and properties of valuations, we have
	\begin{equation}\label{eq:valuationcomputation}
		\ord_{[F]}(\deg(\ell^d)) \ge \min_{G \textrm{ facet of } \Delta} \left(d\ord_{[F]}(X_{G,1} + \dotsb  + X_{G,d}) - \ord_{[F]}([G])  - \sum_{m = 1}^d \ord_{[F]}(X_{G,m}) \right ),
	\end{equation}
	with equality if the minimum is achieved only once. 
	As $X_{G, m}$, $[G]$, and $[F]$ are irreducible polynomials of the same degree which are not scalar multiples of each other (except that $[G] = [F]$ if $G = F$), we see that for $G \not= F$, the quantity in the minimum  in \eqref{eq:valuationcomputation} 
	is nonnegative. Note that $\ord_{[F]}(X_{F, 1} + \dotsb + X_{F, d}) = 0$ by 
	Lemma~\ref{lem:ordAiszero}. Therefore
	the quantity in the minimum 
	in \eqref{eq:valuationcomputation} 
	is equal to $-1$ when $G = F$, and so 
	the minimum is $-1$ and is achieved exactly once. 
\end{proof}

\begin{proof}[Proof of Theorem~\ref{thm:i=0}]
	By Proposition~\ref{prop:h0facetcase}, it suffices to show that if $F$ is a subset of $V$ of size $d$ which is not a facet, then $\ord_{[F]}(\deg(\ell^d)) = 0$.
	By Lemma~\ref{lem:degmu} and Remark~\ref{r:ordernonfaces}, 
	it is enough to show that 
	$\deg_F(\ell^d) = \ev_{\theta_F}(\deg(\ell^d)) \not=0$. 
	
	First 
	assume
	there is a facet $F'$ of $\Delta$ with $|F' \cap F| \le d-2$. Let $\overline{[F']} = \ev_{\theta_F}([F'])$, which is irreducible by Lemma~\ref{lemma:irreducible}. 
	We use Corollary~\ref{cor:KaruXiaoformu} to compute that $\ord_{\overline{[F']}}(\deg_F(\ell^d))$ is bounded below by 
	\begin{equation}\label{eq:valuationcomputation2}
		\min_{G \textrm{ facet of } \Delta} \left(d\ord_{\overline{[F']}}(X_{G, \theta_F, 1} + \dotsb  + X_{G, \theta_F, d}) - \ord_{\overline{[F']}}(\ev_{\theta_F}([G]))  - \sum_{m = 1}^d \ord_{\overline{[F']}}(X_{G, \theta_F, m}) \right ),
	\end{equation}
	with equality if the minimum is achieved only once. 
	If $G \not= F'$, then 
	 $\ord_{\overline{[F']}}(\ev_{\theta_F}([G])) = 0$, because $\overline{[F']}$ is an irreducible polynomial of the same degree as $\ev_{\theta_F}([G])$, and they are not scalar multiples. Similarly, $\ord_{\overline{[F']}}(X_{G, \theta_F, m}) = 0$ (this holds even if $G = F'$). So if $G \not= F'$, then the 
	quantity in the minimum in \eqref{eq:valuationcomputation2} is nonnegative. 
	
	If $G = F'$, then $\ord_{\overline{[F']}}(\ev_{\theta_F}([G])) = \ord_{\overline{[F']}}(\overline{[F']}) = 1$. 
	Write $F' = \{ j_1 < \cdots < j_d \}$ and fix $1 \le m \le d$.  Then for $1 \le m' \le d$, the coefficient  of the monomial $a_{1,0} a_{2,j_1}\cdots a_{m,j_{m-1}} a_{m +1, j_{m + 1}} \cdots a_{d,j_d}$ in $X_{F', \theta_F, m'}$ is $(-1)^m$ if $m = m'$ and is  $0$ otherwise, and  the coefficient of this monomial in $[F']$ is zero. We deduce that
	$X_{F', \theta_F, 1} + \dotsb  + X_{F', \theta_F, d}$ is nonzero with the same degree as $[F']$ and
	$\ord_{\overline{[F']}}(X_{F', \theta_F, 1} + \dotsb  + X_{F', \theta_F, d})  = 0$.
	Therefore, the quantity in the minimum in \eqref{eq:valuationcomputation2} is $-1$ for $G = F'$, 
	so we deduce that $\ord_{\overline{[F']}}(\deg_F(\ell^d)) = -1$. 
	In particular, $\deg_F(\ell^d) \not= 0$. 
	
	Suppose that there is no such facet $F'$.
	Then we show that $\Delta$ must be the suspension $\Sigma$ of the boundary of a $(d-1)$-dimensional simplex, i.e., the case discussed in Section~\ref{sec:special}. 
	Let $v$ be a vertex of $\Delta$ not in $F$, so every facet containing $v$ has $d-1$ vertices from $F$. Let $L$ be the link of $v$. 
	Then $L$  is pure of dimension $d-2$ and has the property that any face of dimension $d-3$ is contained in exactly two faces of dimension $d-2$.  
	Because all facets of $L$ are contained in the boundary of $F$ and $L$ is pure,  $L$ is isomorphic to a subcomplex of $S^{d-2}$.
	The only subcomplex of $S^{d-2}$ which is pure of dimension $d-2$ and has the property  that any face of dimension $d-3$ is contained in exactly two faces of dimension $d-2$ is all of $S^{d-2}$.

	We see that $\Delta$ is isomorphic to the join of $S^{d-2}$ with a disjoint union of some vertices $\{v_1, \dotsc, v_r\}$. 
	Then the link of any facet of $S^{d-2}$ is $\{v_1, \dotsc, v_r\}$. Since every $(d-2)$-dimensional face of $\Delta$ lies in exactly two facets, we must have $r = 2$.
	Therefore $\Delta \cong \Sigma$. The case of $\Sigma$ was treated in Lemma~\ref{lem:HL0sigma}.
\end{proof}

\begin{remark}\label{rem:stronggq=0}
	The above argument, together with the proof of Lemma~\ref{lem:HL0sigma}, shows that if $F$ is a non-face of size $d$, then $\deg_F(\ell^d) = \ev_{\theta_F}(\deg(\ell^d))$ is nonzero. In particular, $\overline{H}_F(\Delta)$ has the strong Lefschetz property in degree $0$ over any field. 
\end{remark}

\section{Almost anisotropy}\label{sec:almostanisotropy}

In this section and the next, we prove Theorem~\ref{thm:strongg} when the characteristic of $k$ is $2$. We use the method introduced by Papadakis and Petrotou \cite{PapadakisPetrotoug} based on the special behavior of differential operators in characteristic $2$, as refined by Karu and Xiao \cite{KaruXiaoAnisotropy}. In the case of $\overline{H}(\Delta)$, their approach establishes that the Hodge--Riemann forms are anisotropic. While, in general, anisotropy fails for $\overline{H}_F(\Delta)$ (see Example~\ref{ex:anisotropyfails}), the same approach allows us to prove anisotropy ``away from a $1$-dimensional subspace'' (Proposition~\ref{prop:almostanisotropy}). Furthermore, as we will show in the next section, this weaker property is enough to deduce the strong Lefschetz property.

\smallskip

Fix a non-face $F$ of $\Delta$ of size $d$.  Recall from the introduction that $\theta_1^F = \sum_{j \not \in F} a_{1,j}x_j$ and that $\overline{H}_F(\Delta)$ is the Gorenstein quotient of $K[\Delta]/(\theta_1^F, \dotsc, \theta_d)$.  
Let $\ell_F = \sum_{j=1}^{n} x_j \in \overline{H}_F^1(\Delta)$. For $0 \le q \le d/2$, define the Hodge--Riemann form $\overline{H}_F^q(\Delta) \times \overline{H}_F^q(\Delta) \to K$ via $(y, z) \mapsto \deg_F(\ell_F^{d - 2q}  y  z)$.
The nondegeneracy of the Hodge--Riemann form
is equivalent to the map $\overline{H}_F^q(\Delta) \to \overline{H}_F^{d  - q}(\Delta)$ given by multiplication by $\ell_F^{d  -2q}$ being an isomorphism, and,  by Lemma~\ref{lem:slp} below, this is equivalent to the strong Lefschetz property in degree $q$.

\begin{lemma}\label{lem:slp}
	Fix some $0 \le q \le d/2$. 
	The algebra $\overline{H}_F(\Delta)$ has the strong Lefschetz property in degree $q$ if and only if multiplication by $\ell^{d-2q}_F$ is an isomorphism from $\overline{H}^{q}_F(\Delta)$ to $\overline{H}^{d - q}_F(\Delta)$, i.e., $\ell_F$ is a strong Lefschetz element. 
\end{lemma}

\begin{proof}
	If $\ell_F$ is a strong Lefschetz element in degree $q$, then clearly $\overline{H}_F(\Delta)$ has the strong Lefschetz property in degree $q$. For the converse, we may replace $k$ by its algebraic closure. 
	There is a nonempty Zariski open subset of $\overline{H}_F^1(\Delta)$ consisting of strong Lefschetz elements in degree $q$. 
	It follows that there is a Zariski open subset $U$ of $k^n$ such that, for every $(\lambda_1, \dotsc, \lambda_n) \in U$, the element $\sum_{j=1}^{n} \lambda_j x_j$ is a strong Lefschetz element in degree $q$. 
	Therefore, we can find a strong Lefschetz element $\ell_{F,\lambda} = \sum \lambda_j x_j$ with each $\lambda_j \in k^{\times}$. Let
	$$H_{F, \lambda}(\Delta) = K[\Delta]/(\sum_{j \not \in F} \lambda_j a_{1,j} x_j, \sum_{j} \lambda_j a_{2,j} x_j, \dotsc, \sum_{j} \lambda_j a_{d,j} x_j),$$
	and let $\overline{H}_{F, \lambda}(\Delta)$ be the Gorenstein quotient. Because the $\lambda_j a_{i,j}$ are algebraically independent, $\ell_{F,\lambda}$ is a strong Lefschetz element for $\overline{H}_{F, \lambda}(\Delta)$. Let $\Phi \colon \overline{H}_F(\Delta) \to \overline{H}_{F, \lambda}(\Delta)$ be the graded isomorphism given by sending $x_j$ to $\lambda_j x_j$. Then we have a commutative square
	\begin{center}
		\begin{tikzcd}
			\overline{H}^q_F(\Delta) \arrow[r, "\ell_F^{d - 2q}"] \arrow[d, "\Phi"]
			& \overline{H}^{d-q}_F(\Delta) \arrow[d, "\Phi"] \\
			\overline{H}^q_{F, \lambda}(\Delta) \arrow[r, "\ell_{F,\lambda}^{d - 2q}"]
			& \overline{H}^{d-q}_{F, \lambda}(\Delta).
		\end{tikzcd}
	\end{center}
	As the bottom horizontal arrow is an isomorphism, so is the top horizontal arrow. 
\end{proof}

A similar equivalence holds for $\overline{H}(\Delta)$, i.e., $\overline{H}(\Delta)$ has the strong Lefschetz property in degree $q$ if and only if $\ell$ is a strong Lefschetz element in degree $q$. 

Say that $\overline{H}_F(\Delta)$ has the \emph{weak Lefschetz property} in degree $q$ if there is an element $y \in \overline{H}^1_F(\Delta)$ such that multiplication by $y$ induces a map of full rank from $\overline{H}_F^q(\Delta)$ to $\overline{H}_F^{q+1}(\Delta)$. It is well-known (see, e.g., \cite[Lemma 2.3 and Remark 2.4]{MZ}) that an artinian Gorenstein ring which is generated in degree $1$ (such as $\overline{H}_F(\Delta)$) has the weak Lefschetz property in all degrees if and only if it has the weak Lefschetz property in middle degree (i.e., degree $\lfloor d/2 \rfloor$); furthermore,
multiplication by $y$ induces an injection from 
$\overline{H}_F^q(\Delta)$ to $\overline{H}_F^{q+1}(\Delta)$
for $q < d/2$ and a surjection for $q > d/2 - 1$. 
The proof of Lemma \ref{lem:slp} gives the following result.

\begin{lemma}\label{lem:wlp}
For any $q$, the algebra $\overline{H}_F(\Delta)$ has the weak Lefschetz property in degree $q$ if and only if multiplication by $\ell_F$ is 
	a map of full rank
	from $\overline{H}^{q}_F(\Delta)$ to $\overline{H}^{q+1}_F(\Delta)$, i.e., $\ell_F$ is a weak Lefschetz element in degree $q$. 
\end{lemma}

\medskip

For the rest of this section, let $k$ be a field of characteristic $2$, and fix a non-face $F$ of $\Delta$ of size $d$.   Recall that by Lemma~\ref{lem:slp}, Theorem~\ref{thm:strongg} holds in degree $q$ if and only if the Hodge--Riemann form on $\overline{H}_F^q(\Delta)$ is nondegenerate. We will show that when the latter condition holds, the induced quadratic form is ``almost'' anisotropic, in the sense that there is at most one nonzero vector up to scaling for which the quadratic form is zero.

Recall that a quadratic form $Q$ on a vector space $V$ is \emph{anisotropic} if $Q(v) \neq 0$  for all nonzero $v$ in $V$.
Fix $0 \le q \le d/2$. 
Karu and Xiao proved the anisotropy of the (quadratic form associated to the)  Hodge--Riemann form on $\overline{H}^q(\Delta)$ \cite[Theorem 4.4]{KaruXiaoAnisotropy}, i.e., $\deg (\ell^{d - 2q}  z^2) \neq 0$ for nonzero $z \in \overline{H}^q(\Delta)$. Our goal is to analogously 
establish ``almost'' anisotropy for the Hodge--Riemann form on $\overline{H}_F^q(\Delta)$. The following example shows that anisotropy need not hold.

\begin{example}\label{ex:anisotropyfails}
Let $d=2$, and consider $\Sigma$ as in Section~\ref{sec:special}, i.e., $\Sigma$ has vertex set $\{1,2,3,4\}$ and minimal non-faces $\{1,2\}$ and $\{3,4\}$. 
Let $F = \{1,2\}$, and consider $H_F(\Sigma) = \overline{H}_F(\Sigma)$. Then $\theta_1^F = a_{1,3} x_3 + a_{1,4} x_4$, so the relation $x_3  x_4 = 0$ in $K[\Sigma]$ implies that $x_3^2 = 0$ in $H_F(\Sigma)$.
As $x_3 \not= 0$ in $H_F(\Sigma)$, anisotropy fails for $H_F(\Sigma)$. 
\end{example}

Let $W_q \subset \overline{H}_F^q(\Delta)$ be the span  of all monomials whose support is not contained in $F$. 
We will mainly be interested in the case when $q = \lfloor d/2 \rfloor$.

	\begin{lemma}\label{lem:codim1}
		The codimension of $W_q$ in $\overline{H}^{q}_F(\Delta)$ is at most $1$. 
		Consider a linear form $g = \sum_{i = 1}^n \lambda_i x_i$ for some $\lambda_i \in k$ such that $\lambda_v \neq 0$ for some $v \in F$. If 
		$g^q$ lies in $W_q$, then 
		$W_q = \overline{H}^{q}_F(\Delta)$. 
	\end{lemma}
\begin{proof}
Let $w$ be a vertex not in $F$. We claim that $\{ g \} \cup \{ x_v : v \notin F \cup \{ w \}\}$ is a basis of 
	$H^1_F(\Delta)$. Assuming this claim, monomials  of degree $q$ in this basis span $\overline{H}^{q}_F(\Delta)$. Every such monomial except $g^{q}$ lies in $W_q$, proving that $W_q$ has codimension at most $1$, and that $W_q = \overline{H}^{q}_F(\Delta)$ if $g^{q}$ lies  in $W_q$. 
	
	It remains to establish the claim. We need to show that 
	$\theta_1^F, \theta_2, \ldots, \theta_d$ together with 
	$\{ g \} \cup \{ x_v : v \notin F \cup \{ w \}\}$ are linearly independent, and hence form a basis of $K[x_1,\ldots,x_n]_1$. Without loss of generality, assume that $F = \{ 1,\ldots, d\}$ and $w = d + 1$. 
	Let $M$ denote 
    the $(d + 1) \times (d + 1)$ matrix whose rows record the coefficients of
    $\theta_1^F, \theta_2, \ldots, \theta_d, g$
    with respect to the vertices $\{ 1,\ldots,d + 1\}$. 
    It is enough to show that $\det M \neq 0$. 
		Let $M'$ be the submatrix obtained by removing the first row and last column. 
		Since the only nonzero entry in the first row of $M$ is the last entry, it suffices to show
    that $\det M' \neq 0$. This follows since all entries of $M'$ are generic except the last row, which is nonzero by assumption.
\end{proof}

Let $W_q^{\perp} = \{ z \in \overline{H}^q_F(\Delta) : \deg_F(\ell_F^{d - 2q}  z  w) = 0  \textrm{ for all } w \in W_q \}$. If we assume the Hodge--Riemann form on $\overline{H}^q_F(\Delta)$ is nondegenerate in degree $q$,  then 
$W_q^{\perp}$ is the orthogonal complement of $W_q$ 
and 
the dimension of $W_q^{\perp}$ is at most $1$ by Lemma~\ref{lem:codim1}. 

\begin{example}
	In Proposition \ref{prop:degree1}, we will show that $\overline{H}^1_F(\Delta) = H^1_F(\Delta)$. In this case, as $\theta_1^F$ gives a linear relation between $\{ x_w : w \notin F \}$ in $H^1_F(\Delta)$, $W_1$ has codimension $1$. In Example~\ref{ex:anisotropyfails}, $W_1 = W_1^\perp = \operatorname{Span}(x_3) \subset H^1_F(\Delta)$. 
\end{example}
 
We have the following application of Corollary~\ref{cor:karuxiaoanalogue2}.  
 
\begin{proposition}\label{prop:almostanisotropy}
Let $z \in \overline{H}^{q}_F(\Delta)$. If $\deg_F(\ell_F^{d - 2q}  z^2) = 0$, then $z \in W_q^{\perp}$. 
\end{proposition} 
\begin{proof}
Let $J = (j_1,\ldots,j_q)$ be a sequence of elements of $V$ such that the support of $x_J$ is not contained in $F$. We may assume that $j_1 \notin F$. If we can show that $\deg_F(\ell_F^{d - 2q} x_J z) = 0$, then applying this to all such $J$ implies that $z \in W_q^{\perp}$.

Let $L$ be a sequence of elements of $V$ of size $\lfloor d/2 \rfloor - q$, and let $v$ be a vertex of $V$. 
Let $(J_1 | \cdots | J_r)$ denote the concatenation of sequences of vertices $J_1,\ldots,J_r$. Define 
\[
I = \begin{cases}
	(J | J | L | L ) &\textrm{ if } d \textrm{ is even} \\
	(J | J | L | L | v ) &\textrm{ if } d \textrm{ is odd.} \\
\end{cases}
\]
If $d$ is even, then Corollary~\ref{cor:karuxiaoanalogue2} implies that
\[
0 = \partial_I \deg_F(\ell_F^{d - 2q}  z^2) = \deg_F(\ell_F^{d/2 - q} x_J x_L z)^2.
\]
Since this holds for all $L$, we deduce that $\deg_F(\ell_F^{d - 2q} x_J z) = 0$. 

If $d$ is odd, then Corollary~\ref{cor:karuxiaoanalogue2} implies that
\[
0 = \partial_I \deg_F(\ell_F^{d - 2q}  z^2) = \sum_{i = 1}^n \deg_F(\ell_F^{(d - 1)/2 - q} x_J x_L \sqrt{x_i x_v} z)^2 = \deg_F(\ell_F^{(d - 1)/2 - q} x_J x_L x_v z)^2.
\]
Since this holds for all $L$ and $v$, we deduce that $\deg_F(\ell_F^{d - 2q} x_J z) = 0$. 
\end{proof}

\begin{corollary}\label{cor:anisocriterion}
	Assume that either $W_q^{\perp} = 0$, 
	 or the Hodge--Riemann form on $\overline{H}_F^q(\Delta)$ is nondegenerate and $W_q^{\perp} \not \subset W_q$.  Then the Hodge--Riemann form on $\overline{H}_F^q(\Delta)$ is anisotropic. 
\end{corollary}
\begin{proof}
		If $W_q^{\perp} = 0$, then Proposition~\ref{prop:almostanisotropy} implies that this form is anisotropic. Assume that the Hodge--Riemann form on $\overline{H}_F^q(\Delta)$ is nondegenerate and $W_q^{\perp} \not \subset W_q$. Then $\dim W_q^{\perp} = \codim W_q = 1$ by Lemma~\ref{lem:codim1}. Let $W_q^{\perp} = \spn(z)$, and assume that $\deg_F(\ell_F^{d - 2q} z^2) = 0$. Then 
		$z \in (W_q^\perp)^{\perp} = W_q$, contradicting our assumption that $W_q^\perp$ is not contained in $W_q$. 
		Hence $\deg_F(\ell_F^{d - 2q} z^2) \neq 0$, and the result follows from 
	 Proposition~\ref{prop:almostanisotropy}.
	 \end{proof}

\begin{proposition}\label{prop:degencriterioneven}
	Assume that the Hodge--Riemann form on $\overline{H}_F^q(\Delta)$ is nondegenerate. 
	Let $U$ be a subspace of $\overline{H}^{q}_F(\Delta)$ where the restriction of the Hodge--Riemann form on $\overline{H}_F^q(\Delta)$ to $U$ is degenerate. 
    Then $U \subset W_q$. 
\end{proposition}
\begin{proof}
	If $W_q^{\perp} = 0$, then Corollary~\ref{cor:anisocriterion} implies that the Hodge--Riemann form in degree $q$ is anisotropic, and so the restriction to any subspace is anisotropic and therefore nondegenerate. By Lemma~\ref{lem:codim1}, we may therefore assume that $W_q^{\perp}$ is one-dimensional, generated by $y_1$. By Corollary~\ref{cor:anisocriterion}, we may assume that $y_1 \in W_q$. By Proposition~\ref{prop:almostanisotropy}, if $y_1 \notin U$, then the restriction of  the Hodge--Riemann form  to $U$ is anisotropic and hence nondegenerate, a contradiction. Therefore 
	$U$ must contain $y_1$. 
	
	Assume that $U$ is not contained in $W_q$. 
	Then the codimension of $U \cap W_q$ in $U$ is $1$, and so we may extend $y_1$ to a basis $y_1, y_2, \dotsc, y_r$ of $U$  with $y_1, \dotsc, y_{r-1} \in W_q$. 
	Since $y_1 \in W_q^\perp$,  $\deg_F(\ell_F^{d - 2q}  y_1  y_i) = 0$ for $1 \le i < r$. 
	As $(W_q^{\perp})^{\perp} = W_q$ and $y_r \not \in W_q$, it follows that $\deg_F(\ell_F^{d - 2q}  y_1  y_r) \not= 0$. Let $M$ be the $r \times r$ matrix whose $(i, j)$th entry is $\deg_F(\ell_F^{d - 2q}  y_i  y_j)$. Let $M'$ be the $(r-2) \times (r-2)$ submatrix given by rows $\{2, \dotsc, r-1\}$ and columns $\{2, \dotsc, r-1\}$. We see that 
	$$\det(M) = \deg_F(\ell_F^{d - 2q}  y_1  y_r)^2  \det(M').$$
	As the span of $y_2, \dotsc, y_{r-1}$ does not contain $W_q^{\perp}$, Proposition~\ref{prop:almostanisotropy} implies that the Hodge--Riemann form  restricted to the span of $y_2, \dotsc, y_{r-1}$ is anisotropic, and hence nondegenerate. In particular, $\det(M') \not=0$. 
	We conclude that $\det(M) \not=0$, contradicting the assumption that the Hodge--Riemann form is degenerate when restricted to $U$. 
\end{proof}

\section{Strong Lefschetz in characteristic 2}\label{sec:char2}

In this section, we complete the proof of the strong Lefschetz property for $\overline{H}_F(\Delta)$ in characteristic $2$. We use some ideas from \cite[Section 9.1]{PapadakisPetrotoug}, which are in turn inspired by ideas of Swartz, e.g., \cite[Proposition 4.24]{Swartz09}.

Throughout this section, we assume that $k$ has characteristic $2$. 
Let $d > 1$ be arbitrary. Let $S(\Delta)$ denote the suspension of $\Delta$, which has vertex set $V \cup \{n+1, n+2\} = \{1, \dotsc, n+1, n+2\}$. Since $\Delta$ is a connected pseudomanifold, so is $S(\Delta)$. Let $\hat{F} = F \cup \{n+1\}$; note that $\hat{F}$ is a non-face of $S(\Delta)$. Let $\hat{K} = k(a_{i,j})_{1 \le i \le d+1, 1 \le j \le n+2}$.

We will consider a slightly different l.s.o.p. for $S(\Delta)$. Set $\hat{\theta}_1^F = a_{d+1, n+1}^{-1} \sum_{j \not \in \hat{F}} a_{1,j} x_j$. For $1 < i \le d$, set
$$\hat{\theta}_i = a_{i,n+1} x_{n+1} + a_{i, n+2}x_{n+2} + \sum_{j=1}^{n} (a_{i,j} + a_{d+1,j} a_{i,n+1} a_{d+1,n+1}^{-1})x_j,$$
and set $\theta_{d+1} = \sum_{j = 1}^{n + 2} a_{d+1, j} x_j$. Let $\hat{\theta}_F = (\hat{\theta}_1^F, \hat{\theta}_2, \dotsc, \hat{\theta}_d, \theta_{d+1})$. Let $\overline{H}_{\hat{\theta}_F}(S(\Delta))$ be the Gorenstein quotient of 
$\hat{K}[S(\Delta)]/\hat{\theta}_F$. 
Note that the coefficients of $\hat{\theta}_F$ are generic except that, when $j \in \hat{F}$, the coefficient of $x_j$ in $\hat{\theta}_1^F$ is $0$. In particular, $\overline{H}_{\hat{\theta}_F}(S(\Delta))$ is isomorphic as a graded ring to $\overline{H}_{\hat{F}}(S(\Delta))$.

We first outline the argument. 
After extending scalars, we can identify $\overline{H}_F(\Delta)$ with the ideal $(x_{n+1})$ in $\overline{H}_{\hat{\theta}_F}(S(\Delta))$ (see Proposition~\ref{prop:gysin}). When $d$ is odd, we will show that the weak Lefschetz property for $\overline{H}_F(\Delta)$ is equivalent to the Hodge--Riemann form on $\overline{H}^{(d+1)/2}_{\hat{\theta}_F}(S(\Delta))$ being nondegenerate when restricted to the ideal $(x_{n+1})$. 
We then apply the results of Section~\ref{sec:almostanisotropy} to  $\overline{H}_{\hat{\theta}_F}(S(\Delta))$ to deduce the latter result, and hence 
deduce the weak Lefschetz property for $\overline{H}_F(\Delta)$. We can then use the results of Section~\ref{sec:almostanisotropy} on $\overline{H}_F(\Delta)$ to deduce the strong Lefschetz property. The case when $d$ is even is similar, using that $d+1$ is odd, so we already know the weak Lefschetz property for $\overline{H}_{\hat{\theta}_{F}}(S(\Delta))$.

In the statement of the proposition below, we extend $\deg_F$ to an isomorphism of $\hat{K}$-vector spaces from $\overline{H}_F^{d}(\Delta) \otimes_K \hat{K}$ to $\hat{K}$.

\begin{proposition}\label{prop:gysin}
	There is an isomorphism of $\hat{K}$-algebras
	$$\overline{\varphi} \colon \overline{H}_{\hat{\theta}_F}(S(\Delta))/\operatorname{ann}(x_{n+1}) \stackrel{\sim}{\to} \overline{H}_F(\Delta)\otimes_K \hat{K}$$
	given by $\overline{\varphi}(x_i) = x_i$ for $i \le n$, $\overline{\varphi}(x_{n+1}) = \sum_{j=1}^{n} a_{d+1, j} a_{d+1, n+1}^{-1} x_j$, and $\overline{\varphi}(x_{n+2}) = 0$. Let $\varphi$ be the composition 
	$$\overline{H}_{\hat{\theta}_F}(S(\Delta)) \to \overline{H}_{\hat{\theta}_F}(S(\Delta))/\operatorname{ann}(x_{n+1}) \stackrel{\overline{\varphi}}{\to} \overline{H}_F(\Delta) \otimes_K \hat{K}.$$
	Then $\deg_{\hat{\theta}_F}(z  x_{n+1}) = \deg_{F}(\varphi(z))$ for all $z \in \overline{H}^d_{\hat{\theta}_F}(S(\Delta))$. 
\end{proposition}

\begin{proof}
	First observe that there is a map $H_{\hat{\theta}_{F}}(S(\Delta)) \to H_F(\Delta) \otimes_K \hat{K}$ defined by 
	$x_i \mapsto x_i$ for $i \le n$ and $x_{n+2} \mapsto 0$. As $\theta_{d+1} = 0$, this implies that  $x_{n+1}$ is sent to $a_{d+1, n+1}^{-1} \sum_{j=1}^{n} a_{d+1, j}  x_j$.  We check that the induced map $\varphi' \colon H_{\hat{\theta}_{F}}(S(\Delta)) \to \overline{H}_F(\Delta) \otimes_K \hat{K}$ descends to $\overline{H}_{\hat{\theta}_{F}}(S(\Delta))/\operatorname{ann}(x_{n+1})$. For this, we need to show that if we have $0 \le q \le d$ and $y \in H^q_{\hat{\theta}_{F}}(S(\Delta))$ with $\deg_{\hat{\theta}_F}(x_{n+1}  y  z) = 0$ for all $z \in H_{\hat{\theta}_{F}}^{d - q}(S(\Delta))$, then
	 $\varphi'(y) = 0$. 
	 To do so, it suffices to prove that
	$$\deg_{\hat{\theta}_F}(x_{n+1}  w) = \deg_F(\varphi'(w)) \quad \text{for all }w \in H^d_{\hat{\theta}_{F}}(S(\Delta)).$$
	
	We claim that $H^d_{\hat{\theta}_{F}}(S(\Delta))$ is spanned by squarefree monomials in $x_1, \dotsc, x_n$. 
	It is well-known that 
	it is spanned by squarefree monomials in $x_1,\ldots, x_n, x_{n+1}, x_{n+2}$ supported on $(d-1)$-dimensional faces of $S(\Delta)$; this follows 	 by using \eqref{eq:makesquarefree} to repeatedly increase the size of the support. Let $G=\{i_1,\ldots, i_{d-1}, i_d\}$ be one such face with $i_d\in\{n+1,n+2\}$ and let $G'= (G\smallsetminus \{i_d\}) \cup \{n+1,n+2\} \subset \{1,\ldots,n+2\}$. 
	Then $G'$ has size $d+1$, and 
	Lemma~\ref{lem:Gaussianelim} 
	allows us to express $x_{i_d}\in H^1_{\hat{\theta}_F}(S(\Delta))$ as a linear combination of variables in $\{x_p : p\notin G'\}$. Substituting such an expression for $x_{i_d}$ in the monomial $x_G$ shows that $x_G$ lies in the span of squarefree monomials in $x_1,\ldots,x_n$, as claimed.
	
	For each squarefree monomial $x_{j_1} \dotsb x_{j_d}$ in $x_1, \dotsc, x_n$, 
	 $\deg_{\hat{\theta}_F}(x_{j_1} \dotsb x_{j_d}  x_{n+1})$ equals
	$$
	\det \begin{pmatrix}a_{1, j_1}a_{d+1, n+1}^{-1} & \dotsb & a_{1, j_d}a_{d+1, n+1}^{-1} & 0 \\ a_{2,j_1} + a_{d+1,j_1} a_{2,n+1} a_{d+1,n+1}^{-1} & \dotsb &  a_{2,j_d} + a_{d+1,j_d} a_{2,n+1} a_{d+1,n+1}^{-1} & a_{2, n+1} \\ \vdots & \dotsb & \vdots & \vdots \\ a_{d+1, j_1} & \dotsb & a_{d+1, j_d} & a_{d+1, n+1} \end{pmatrix}^{-1},$$
	where if $j_p \in F$ we interpret $a_{1,j_p}$ as $0$. Using row operations to make the last column $0$ except for the $(d+1, d+1)$ entry, we see that this is equal to $\deg_F(x_{j_1} \dotsb x_{j_d})$.
	We conclude that $\overline{\varphi}$ is well-defined. 

	Multiplication by $x_{n + 1}$ induces a degree $1$ graded isomorphism 
	$\overline{H}_{\hat{\theta}_F}(S(\Delta))/\operatorname{ann}(x_{n+1}) \cong (x_{n + 1}) \subset \overline{H}_{\hat{\theta}_F}(S(\Delta))$. It follows that 
	$\overline{H}_{\hat{\theta}_F}(S(\Delta))/\operatorname{ann}(x_{n+1})$ is an artinian Gorenstein algebra of socle degree $d$. 
	Since $\overline{\varphi}$ is a surjective graded map between artinian Gorenstein algebras whose socles are in the same degree, it is an isomorphism. 
\end{proof}

Proposition~\ref{prop:gysin} then implies that the ideal $(x_{n+1})$ in $\overline{H}_{\hat{\theta}_F}(S(\Delta))$ can be identified with $\overline{H}_F(\Delta) \otimes_K \hat{K}[-1]$. 
Following the strategy of \cite[Section 9.1]{PapadakisPetrotoug}, we apply the results in Section~\ref{sec:almostanisotropy} to $S(\Delta)$ in order to prove the weak Lefschetz property for $\Delta$. 

\begin{proposition} \label{prop:odd-weak-Lefshetz}
	Let $F$ be a non-face of $\Delta$ of size $d$. 	
	Then $\overline{H}_F(\Delta)$ has the weak Lefschetz property. 
\end{proposition}

\begin{proof}
	We can check this after extending scalars to $\hat{K}$; see, e.g., \cite[Proposition 13.3]{PapadakisPetrotoug}.
		We assume the setup of Proposition~\ref{prop:gysin}, and we
	set $\varphi(x_{n+1}) = x$. 	Let $U_{q}$ denote the degree $q$ part of the ideal $(x_{n + 1})$ in $\overline{H}_{\hat{\theta}_F}(S(\Delta))$.  Let $W_{q}$ denote the subspace of $\overline{H}^{q}_{\hat{\theta}_F}(S(\Delta))$ spanned by monomials whose support is not contained in $\hat{F}$.
	
	First assume that $d$ is odd. 
	Because $\overline{H}_F(\Delta)$ is generated in degree $1$, it suffices to prove that multiplication by $x$ induces an isomorphism from $\overline{H}^{(d-1)/2}_F(\Delta) \otimes_K \hat{K}$ to $\overline{H}^{(d+1)/2}_F(\Delta) \otimes_K \hat{K}$. 
	This, in turn, is equivalent to the form $(y, z) \mapsto\deg_F(x  y  z)$ on $\overline{H}^{(d-1)/2}_F(\Delta) \otimes_K \hat{K}$ being nondegenerate.
	By Proposition~\ref{prop:gysin}, 
	$$\deg_F(x  y  z) = \deg_{\hat{\theta}_F}(x_{n+1}^2  \tilde{y}  \tilde{z}),$$
	where $\tilde{y}$ and $\tilde{z}$ are any lifts of $y$ and $z$ to $\overline{H}_{\hat{\theta}_F}(S(\Delta))$. 
	Recall that we can identify  $\overline{H}_F^{(d - 1)/2}(\Delta) \otimes_K \hat{K}$ with $U_{(d + 1)/2}$ such that $y$ and $z$ correspond to $\tilde{y}  x_{n + 1}$ and $\tilde{z}  x_{n + 1}$, respectively.
	Thus, we can identify the form $(y, z) \mapsto\deg_F(x  y  z)$ on $\overline{H}^{(d-1)/2}_F(\Delta) \otimes_K \hat{K}$ with the form 
	$(y', z') \mapsto \deg_{\hat{\theta}_F}(y'  z')$ on $U_{(d + 1)/2}$, i.e., the restriction of the Hodge--Riemann form on 
	$\overline{H}^{(d + 1)/2}_{\hat{\theta}_F}(S(\Delta))$ to $U_{(d + 1)/2}$. Our goal is to show that this form is nondegenerate.
	 
Because $\overline{H}_{\hat{\theta}_F}$ is artinian Gorenstein and $(d+1)/2$ is the middle degree, the Hodge--Riemann form on $\overline{H}^{(d + 1)/2}_{\hat{\theta}_F}(S(\Delta))$ is nondegenerate. Assume that its restriction to $U_{(d + 1)/2}$ is degenerate. Then Proposition~\ref{prop:degencriterioneven} implies that $U_{(d + 1)/2} \subset W_{(d + 1)/2}$.
As $x_{n+1}^{(d+1)/2} \in U_{(d + 1)/2}$, Lemma \ref{lem:codim1} implies that $W_{(d + 1)/2} =   \overline{H}^{(d+1)/2}_{\hat{\theta}_F}(S(\Delta))$, so Corollary \ref{cor:anisocriterion} implies that the Hodge--Riemann form on  $\overline{H}^{(d+1)/2}_{\hat{\theta}_F}(S(\Delta))$ is anisotropic, contradicting the assumption that the restriction to $U_{(d + 1)/2}$ is degenerate. 

	Next assume that $d$ is even. The argument is similar to the one above. 
	Let $\ell_{\hat{F}}$ be the image of $\sum_{i = 1}^{n + 2} x_i$ in $\overline{H}^1_{\hat{\theta}_F}(S(\Delta))$.
	Because $\overline{H}_F(\Delta)$ is generated in degree $1$, it suffices to prove that multiplication by $x$ induces an injection from $\overline{H}^{d/2 - 1}_F(\Delta) \otimes_K \hat{K}$ to $\overline{H}^{d/2}_F(\Delta) \otimes_K \hat{K}$. 
	This follows if we can show that multiplication 
	by $x  \varphi(\ell_{\hat{F}})$ is injective, or, equivalently, that the form 
	$(y, z) \mapsto\deg_F(x  \varphi(\ell_{\hat{F}})  y  z)$ on $\overline{H}^{d/2 - 1}_F(\Delta) \otimes_K \hat{K}$ is nondegenerate. 	By Proposition~\ref{prop:gysin}, 
	$$\deg_F(x   \varphi(\ell_{\hat{F}})  y  z) = \deg_{\hat{\theta}_F}(x_{n+1}^2  \ell_{\hat{F}}   \tilde{y}  \tilde{z}),$$
	where $\tilde{y}$ and $\tilde{z}$ are any lifts of $y$ and $z$ to $\overline{H}_{\hat{\theta}_F}(S(\Delta))$. 
	Recall that we can identify  $\overline{H}_F^{d/2 - 1}(\Delta) \otimes_K \hat{K}$ with $U_{d/2}$ such that $y$ and $z$ correspond to $\tilde{y}  x_{n + 1}$ and $\tilde{z}  x_{n + 1}$ respectively.
	Thus, we can identify the form $(y, z) \mapsto\deg_F(x   \varphi(\ell_{\hat{F}})   y  z)$ on $\overline{H}^{d/2 - 1}_F(\Delta) \otimes_K \hat{K}$ with the form 
	$(y', z') \mapsto \deg_{\hat{\theta}_F}(\ell_{\hat{F}}  y'  z')$ on $U_{d/2}$, i.e., 
	the restriction of the Hodge--Riemann form on 
	$\overline{H}^{d/2}_{\hat{\theta}_F}(S(\Delta))$ to $U_{d/2}$. Our goal is to show that this form is nondegenerate.

Since we have established weak Lefschetz for even-dimensional pseudomanifolds, and, in particular, for $S(\Delta)$, Lemma~\ref{lem:slp} implies that
the Hodge--Riemann form on $\overline{H}^{d/2}_{\hat{\theta}_F}(S(\Delta))$  is nondegenerate. The rest of the argument now proceeds just as above. Explicitly, assume that its restriction to $U_{d/2}$ is degenerate. Then Proposition~\ref{prop:degencriterioneven} implies that $U_{d/2} \subset W_{d/2}$.
As $x_{n+1}^{d/2} \in U_{d/2}$, Lemma \ref{lem:codim1} implies that $W_{d/2} =   \overline{H}^{d/2}_{\hat{\theta}_F}(S(\Delta))$, and Corollary \ref{cor:anisocriterion} implies that the Hodge--Riemann form on  $\overline{H}^{d/2}_{\hat{\theta}_F}(S(\Delta))$ is anisotropic, contradicting the assumption that the restriction to $U_{d/2}$ is degenerate.
\end{proof}

We are now ready to prove the strong Lefschetz property for $\overline{H}_F(\Delta)$ in characteristic $2$.

\begin{theorem}\label{thm:stronggpseudo}
	Let $\Delta$ be a connected simplicial 
	pseudomanifold  
	of dimension $d-1$, and let $0 \le q \le d/2$. 
	Assume that $\Char k = 2$.
	Then 
	for every non-face $F$ of size $d$, $\overline{H}_F(\Delta)$ has the strong Lefschetz property in degree $q$. 
\end{theorem}
\begin{proof}
We 	
want to show that the map $ \ell_F^{d-2q} \colon \overline{H}^q_F(\Delta) \to \overline{H}^{d-q}_F(\Delta)$ is an isomorphism. This is equivalent to showing that the Hodge--Riemann form $(y,z)\mapsto\deg_F(\ell_F^{d-2q}  y  z)$ is nondegenerate on $\overline{H}^q_F(\Delta)$. 

Let $m = \lfloor d/2 \rfloor$. By Proposition~\ref{prop:odd-weak-Lefshetz} and Lemma~\ref{lem:wlp}, $\ell_F$ is a weak Lefschetz element,  and so 
the Hodge--Riemann form on $\overline{H}^m_F(\Delta)$ 
is nondegenerate.
Moreover,
$$  \ell_F^{m-q} \colon \overline{H}^q_F(\Delta) \to \overline{H}^{m}_F(\Delta)$$
is injective and compatible with the Hodge--Riemann forms. For example, in the case when $d$ is odd, for all $y,z\in \overline{H}^q_F(\Delta)$ we have
$$
\deg_F(\ell_F^{d-2q}  y  z)=\deg_F(\ell_F (\ell_F^{m-q} y)(\ell_F^{m-q} z)). 
$$
Let $U$ be the degree $m$ part of the ideal $(\ell_F^{m-q})$ in $\overline{H}_F(\Delta)$.
In order to complete the proof,
it suffices to show that the restriction of the Hodge--Riemann form on $\overline{H}^m_F(\Delta)$ to $U$ is nondegenerate. 

Suppose that the restriction to $U$ is degenerate. 
Proposition \ref{prop:degencriterioneven} implies that $U \subset W_m$, where
$W_{m}$ is the subspace of $\overline{H}^{m}_{F}(\Delta)$ spanned by monomials whose support is not contained in $F$.
 Since $\ell_F^{m} \in U$, Lemma \ref{lem:codim1} implies that $W_m =  \overline{H}_F^m(\Delta)$. Corollary \ref{cor:anisocriterion} implies that the Hodge--Riemann form on $\overline{H}_F^m(\Delta)$ is anisotropic, contradicting the assumption that the restriction to $U$ is degenerate.
\end{proof}

\medskip

As Example~\ref{ex:anisotropyfails} shows, the Hodge--Riemann forms on $\overline{H}_F(\Delta)$ are not, in general, anisotropic. We note one case where the proof of Theorem~\ref{thm:stronggpseudo} can be used to deduce anisotropy. 
Examples of simplicial complexes with small non-faces
include, for instance, all flag complexes. In this case, all minimal non-faces have size two.

\begin{proposition}
Let $\Delta$ be a pseudomanifold, and let $k$ be a field of characteristic $2$. Let $F$ be a non-face of $\Delta$ of size $d$, and suppose that $F$ contains a non-face of size at most $m = \lfloor d/2 \rfloor$. Then, for every $0 \le q \le m$, the Hodge--Riemann form on $\overline{H}^q_F(\Delta)$ is anisotropic. 
\end{proposition}

\begin{proof}
	To simplify the notation, assume that $F=\{1,2,\ldots,d\}$ and that $G=\{1,2,\ldots,r\}$ is a non-face of $\Delta$ for some $r \le m$. 
Let $W_m \subset \overline{H}^m_F(\Delta)$ be the span of monomials whose support is not contained in $F$. As in the proof of Theorem~\ref{thm:stronggpseudo}, 
by Lemma \ref{lem:codim1} and Corollary \ref{cor:anisocriterion},
it suffices to show that $x_1^m \in W_m$. 

Using \eqref{eq:gaussianelim}, for each $2\leq i\leq r$, we can express $x_i\in H^1_F(\Delta)$ as the sum of $\lambda_i x_1$ and some linear combination of $\{x_p : d+2\leq p\leq n\}$. Here, up to a sign, 
$\lambda_i$ equals $\ev_{\theta_F}([\{1, \ldots,d+1\} \smallsetminus \{i\}])/\ev_{\theta_F}( [\{2, \ldots, d+1\}])$, and, in particular, is nonzero.
	 Now, since $\{1,2,\ldots,r\}$ is a non-face of $\Delta$, the product $x_1^{m-r+1}x_2\ldots x_r$ is equal to zero. Substituting the above expressions for $x_2,\ldots, x_r$ in this product implies that $\lambda_2\cdots\lambda_r x_1^m$ is a linear combination of monomials in $\{x_p : d+2\leq p\leq n\}$, so $x_1^m$ belongs to $W_m$. 
\end{proof}

\section{Proofs of theorems}\label{sec:proofs}

In this section, we prove Theorem~\ref{thm:strongg}, then Theorem~\ref{thm:oddmultHodgeRiemann}, and then finally Theorem~\ref{thm:middledegree}. 

\medskip

Let $\Delta$ be a connected oriented simplicial $k$-homology manifold of dimension $d-1$. We will need the following result of the second author and Swartz, which uses as input results of Gr\"{a}be and Schenzel  \cite{Grabe,SchenzelNumberFacesSimplicialComplexes}. Let $\beta_q = \dim \tilde{H}_q(\Delta; k)$, the dimension of the reduced homology of $\Delta$ over $k$. By the universal coefficient theorem, this depends only on the characteristic of $k$. 
Let $(h_0(\Delta), \dotsc, h_d(\Delta))$ be the $h$-vector of $\Delta$. 
Let $\overline{H}_{\mu}(\Delta)$ be the Gorenstein quotient of $K[\Delta]/(\mu_1, \dotsc, \mu_d)$ for an l.s.o.p. $\mu = (\mu_1, \dotsc, \mu_d)$ for $K[\Delta]$.

\begin{proposition}\cite[Theorem 1.3 and 1.4]{NovikSwartzGorensteinRings}\label{prop:samedim}
Let $\Delta$ be a connected oriented simplicial $k$-homology manifold of dimension $d-1$.
	Let $\mu = (\mu_1, \dotsc, \mu_d)$ be an l.s.o.p. for $K[\Delta]$.
	Then
	$$\dim \overline{H}^q_{\mu}(\Delta) = \begin{cases} h_q(\Delta) 
		+ \binom{d}{q}	\sum_{p = 0}^{q - 1} (-1)^{q-p} \beta_{p}
		& \text{if } 0 \le q < d \\ 1 & \text{if } q = d. \end{cases}$$
\end{proposition}
In particular, $\dim \overline{H}_{\mu}^q(\Delta)$ is independent of the choice of l.s.o.p. This will be used to compare the Hodge--Riemann forms on $\overline{H}(\Delta)$ to those on $\overline{H}_{\mu}(\Delta)$, for various l.s.o.p.s $\mu$. 

\medskip

\begin{proof}[Proof of Theorem~\ref{thm:strongg}]
When $\Char k = 2$, the result follows immediately from  Theorem~\ref{thm:stronggpseudo}. Now suppose that $\Char k = 0$, and that the integral homology of the link of any face (including the empty face) of $\Delta$ has no $2$-torsion.

	The assumption that the integral homology of the link of every face has no $2$-torsion implies that $\Delta$ is a homology manifold over $\mathbb{F}_2$, and that $\dim H_q(\Delta; k) = \dim H_q(\Delta; \mathbb{F}_2)$ for all $q$.
	Let $\overline{H}_{F, 2}(\Delta)$ be the Gorenstein quotient of $\mathbb{F}_2(a_{i,j})[\Delta]/(\theta_1^F, \dotsc, \theta_d)$. 
By Proposition~\ref{prop:samedim}, $\dim \overline{H}^q_F(\Delta) = \dim \overline{H}^q_{F, 2}(\Delta)$ for each $q$. 
	
	Fix $0 \le q \le d/2$, and let $S = \{m_1, \dotsc, m_p\}$ denote the set of monomials of degree $q$ in $K[\Delta]$. Let $M_0$ be the $p \times p$ matrix whose $(i, j)$th entry is $\deg_F(\ell^{d - 2q}_F  m_i  m_j)$. Then $\overline{H}_F(\Delta)$ has the strong Lefschetz property in degree $q$ if and only if the rank of $M_0$ is equal to $\dim \overline{H}^q_F(\Delta)$. 
	
	Let $M_2$ be the $p \times p$ matrix whose $(i, j)$th entry is $\deg_{F,2}(\ell^{d - 2q}_F  m_i  m_j)$, i.e., the degree in $\overline{H}_{F, 2}(\Delta)$. Because $\overline{H}_{F, 2}(\Delta)$ has the strong Lefschetz property by Theorem~\ref{thm:stronggpseudo}, the rank of $M_2$ is equal to $\dim \overline{H}^q_F(\Delta)$.

	By Lemma~\ref{lem:degmu},
	$\deg_F(\ell^{d - 2q}_F  m_i  m_j) = \ev_{\theta_F}(\deg(\ell^{d - 2q}  m_i  m_j))$ lies in the localization of $\Z[a_{i,j}]$ at the polynomials  $\{ \ev_{\theta_F}([G]) : G \textrm{ is a facet of } \Delta \}$. Note 
	 that $\deg_{F,2}(\ell^{d - 2q}_F  m_i  m_j)$ is obtained by reducing $\deg_F(\ell^{d - 2q}_F  m_i  m_j)$ modulo $2$. 
	Hence $\dim \overline{H}^q_F(\Delta) \ge \rank (M_0) \ge \rank (M_2) = \dim \overline{H}^q_{F,2}(\Delta)$. Since $\dim \overline{H}^q_F(\Delta) = \dim \overline{H}^q_{F, 2}(\Delta)$, 
	the rank of $M_0$ is equal to $\dim \overline{H}^q_F(\Delta)$, as desired. 
\end{proof}

\begin{remark}\label{rem:polytopal}
When $k$ has characteristic $0$ and $\Delta$ is a polytopal sphere, Theorem~\ref{thm:strongg} can be deduced from Stanley's proof of this case of the algebraic $g$-conjecture \cite{Stanleyg}. Indeed, we can assume $k = \mathbb{Q}$ and choose a realization of $\Delta$ as the boundary of a convex polytope $P$ in $\mathbb{R}^d$ whose vertices are rational. Because $F$ is a non-face, using an affine transformation, we may assume that the vertices of $F$ are contained in the hyperplane where the first coordinate of $\R^d$ vanishes and that the origin is in the interior of $P$.
Then the strong Lefschetz property for $H_F(\Delta) = \overline{H}_F(\Delta)$ follows from the Hard Lefschetz theorem applied to the projective toric variety corresponding to the fan over $P$. 
\end{remark}

\medskip

We now begin the proof of Theorem~\ref{thm:oddmultHodgeRiemann}. 
Let $\mu = (\mu_1,\ldots, \mu_d)$ be an l.s.o.p. for $K[\Delta]$.
Let $\ell_\mu$ denote the image of $\sum_{j = 1}^n x_j$ in $\overline{H}^1_\mu(\Delta)$. 
Recall that $R \subset K$ denotes the localization of $k[a_{i,j}]$ at the irreducible polynomials  $\{ [G] : G \textrm{ facet of } \Delta \}$, and $\ev_\mu \colon R \to K$ is the map defined by $\ev_{\mu}(a_{i,j}) = \mu_{i,j}$.

\begin{lemma}\label{lem:ordP}
Let $\Delta$ be a connected oriented simplicial $k$-homology manifold. 
	Let $\mu = (\mu_1,\ldots, \mu_d)$ be an l.s.o.p.
	and let $0 \le q \le d/2$.
	Suppose that 
	multiplication by $\ell^{d-2q}_\mu$ is an isomorphism from $\overline{H}^{q}_\mu(\Delta)$ to $\overline{H}^{d - q}_\mu(\Delta)$, i.e., $\ell_\mu$ is a strong Lefschetz element in degree $q$. 
	Let
	$P \in k[a_{i,j}]$ be an irreducible polynomial such that $\ev_\mu(P) = 0$. 
	Then there are monomials $y_1, \dotsc, y_p$ 
	whose images 
	form a basis 
	of  $\overline{H}^q(\Delta)$
	and $\ord_{P}(\det M) = 0$, where
	$M$ is the $p \times p$ matrix whose $(i, j)$ entry is
	$\deg(\ell^{d - 2q}  y_i  y_j)$. 	
	In particular,
	$\overline{H}(\Delta)$ has the strong Lefschetz property in degree $q$, and,
	 if $D_{q} \in K^{\times}/(K^{\times})^2$ is the determinant of the Hodge--Riemann form on $\overline{H}^q(\Delta)$, then
	$\ord_P(D_q) = 0 \in \Z/2\Z$. 
\end{lemma}
\begin{proof}
	Choose monomials $y_1, \dotsc, y_p$ in the degree $q$ part of $K[\Delta]$ such that their images form
	a basis for $\overline{H}^q_\mu(\Delta)$; this is possible because $\overline{H}^q_\mu(\Delta)$ is spanned by the images of monomials.
	
	Let $M$ be the $p \times p$ matrix whose $(i, j)$ entry is $\deg(\ell^{d - 2q} y_i  y_j)$.
	By Lemma~\ref{lem:degmu}, each entry of $M$ lies in $R$, so $\det M$ lies in $R$. By Remark~\ref{r:ordernonfaces}, if we can show that $\ev_{\mu}(\det M) \neq 0$, then 
	$\ord_{P}(\det M) = 0$.
	
	Let $M_\mu$ be the $p \times p$ matrix whose $(i, j)$ entry is 
		$\deg_\mu(\ell_\mu^{d - 2q}  y_i y_j)$. 
	By Lemma~\ref{lem:degmu}, the $(i, j)$ entry  of $M_\mu$ is
		 $\ev_{\mu}(\deg(\ell^{d - 2q}  y_i y_j))$,
	 so $\det M_\mu = \ev_{\mu}(\det M)$. 
	As $\ell_\mu$ is a strong Lefschetz element for $\overline{H}_\mu^q(\Delta)$, 
	$\det M_\mu \neq 0$, and we conclude that $\ord_{P}(\det M) = 0$.
	
	Finally, that $\det M$ is nonzero implies that
		the images of $y_1,\ldots,y_p$ are 
		 linearly independent in $\overline{H}^q(\Delta)$. As $\dim \overline{H}^q(\Delta) = \dim \overline{H}_\mu^q(\Delta)$ by Proposition~\ref{prop:samedim}, 
		 $\{ y_1,\ldots,y_p \}$ is a basis for $\overline{H}^q(\Delta)$, so $\det M$ computes the determinant of the Hodge--Riemann form on $\overline{H}^q(\Delta)$. 
		 	In particular, since this determinant is  nonzero, $\overline{H}(\Delta)$ has the strong Lefschetz property in degree $q$.
		 This completes the proof.
\end{proof}

\begin{proposition}\label{prop:mult0}
Let $\Delta$ be a connected oriented simplicial $k$-homology manifold. 
	Let $F$ be a subset of $V$ of size $d$ which is not a facet, and let $0 \le q \le d/2$. 
	Suppose that  $\overline{H}_F(\Delta)$ has the strong Lefschetz property in degree $q$. 
		Then $\overline{H}(\Delta)$ has the strong Lefschetz property in degree $q$.
	Let $D_{q} \in K^{\times}/(K^{\times})^2$ be the determinant of the Hodge--Riemann form on $\overline{H}^q(\Delta)$. Then $\ord_{[F]}(D_{q}) = 0$. 
\end{proposition}

\begin{proof}
	By Lemma~\ref{lem:slp}, $\ell_F$ is a strong Lefschetz element for $\overline{H}^q_F(\Delta)$. 
	The result now follows by applying Lemma~\ref{lem:ordP} with $\mu = \theta_F$ and $P = [F]$. 
\end{proof}

\begin{lemma}\label{lem:disjoint}
	
		Let $F$ be a subset of $V$ of size $d$.
	Then for each $q$, there is a basis for $\overline{H}^q(\Delta)$ consisting of the images of monomials in $K[\Delta]$ whose support is disjoint from $F$. 
\end{lemma}

\begin{proof}
	By Lemma~\ref{lem:Gaussianelim}, 
	one can write any monomial in $\overline{H}^1(\Delta)$ in terms of the monomials corresponding to vertices not in $F$. As $\overline{H}(\Delta)$ is generated in degree $1$, this implies that each $\overline{H}^q(\Delta)$ is spanned by monomials whose support is disjoint from $F$. Some subset of these monomials forms a basis. 
\end{proof}

For a facet $F$ of $\Delta$, let $\Delta'$ be the simplicial complex obtained by doing a stellar subdivision in the interior of $F$, i.e., the vertex set of $\Delta'$ is $V \cup \{n+1\} = \{1, \dotsc, n+1\}$, and the facets of $\Delta'$ are the facets of $\Delta$ except for $F$, together with $(F \cup \{n+1\}) \smallsetminus \{j\}$ for each $j \in F$. Then $\Delta'$ is a connected oriented simplicial $k$-homology manifold 
with its orientation 
determined
by orienting the facets of $\Delta'$ which are also facets of $\Delta$ in the same way that they are oriented in $\Delta$.

By Proposition~\ref{prop:mult0},
if $F$ is a non-face of size $d$ and $\overline{H}_F(\Delta)$ has the strong Lefschetz property in degree $q$, then so does $\overline{H}(\Delta)$. If $\Delta$ has no non-faces of size $d$, then $\Delta$ must be isomorphic to $S^{d-1}$, and so Theorem~\ref{thm:oddmultHodgeRiemann} holds for $\Delta$ by Example~\ref{ex:boundarysimplex}. When proving Theorem~\ref{thm:oddmultHodgeRiemann}, we may therefore assume that $\overline{H}(\Delta)$ has the strong Lefschetz property in degree $q$. 

\begin{proof}[Proof of Theorem~\ref{thm:oddmultHodgeRiemann}]
	As Theorem~\ref{thm:i=0} implies Theorem~\ref{thm:oddmultHodgeRiemann} when $q = 0$, we may assume that $0 < q \le d/2$. 
	
Theorem~\ref{thm:strongg} and Proposition~\ref{prop:mult0} show that if $F$ is not a facet, then $\ord_{[F]}(D_{q}) = 0$. Suppose that $F$ is a facet of $\Delta$. By Lemma~\ref{lem:disjoint}, we may choose a collection of monomials $y_1, \dotsc, y_p \in K[\Delta]$ of degree $q$ whose support is disjoint from $F$ and such that their image in $\overline{H}^{q}(\Delta)$ 
is a basis. By the version of Lemma~\ref{lem:slp} for $\overline{H}(\Delta)$, $\ell$ is a strong Lefschetz element in degree $q$. 
	Let $M$ be the $p \times p$ matrix whose $(i, j)$th entry is $\deg(\ell^{d - 2q}  y_i  y_j)$, 
	so $M$ is nonsingular and the image of $\det M$ in $K^{\times}/(K^{\times})^2$ is $D_{q}$.
	
	Let $\Delta'$ be the simplicial complex obtained by doing a stellar subdivision in the interior of $F$, with orientation as described above and degree map $\deg'$. 
	We can identify $K[\Delta]/(x_F)$ with a subring of $K[\Delta']$, and hence consider the images $y_1',\ldots,y_p'$ of  $y_1,\ldots,y_p$ in $K[\Delta']$. Set $y_{p+1}' = x_{n+1}^{q}$. 
	
	We claim that the images of  $y_1', \dotsc, y_{p + 1}'$ span 
$\overline{H}^q(\Delta')$. Indeed, by Lemma~\ref{lem:disjoint}, 
$\overline{H}(\Delta')$ is spanned as a vector space
by monomials whose support is disjoint from $F$. The latter consists of monomials supported away from $F \cup \{ n + 1\}$, together with powers of $x_{n + 1}$. Let $z$ be a monomial in $K[\Delta]$ of degree $q$ whose support is disjoint from $F \cup \{ n + 1\}$. Then 
$z = \sum_{i = 1}^p \lambda_i y_i$ in $\overline{H}^q(\Delta)$
for some $\lambda_i \in K$. Consider $z' := z - \sum_{i = 1}^p \lambda_i y'_i$ in $\overline{H}^q(\Delta')$. 
By Lemma~\ref{lem:deglocal}, the product of 
$z'$ with any monomial of degree $d - q$ supported away from $F \cup \{ n + 1\}$ is zero. Also, $z' x_{n + 1}^{d - q} = 0$. We deduce that $z' = 0$, and the claim follows.

Let 
$\ell'$ be the image of $\sum_{j = 1}^{n + 1} x_j$ in $\overline{H}^1(\Delta')$,
and let $M'$ be the $(p+1) \times (p+1)$ matrix whose $(i, j)$th entry is 
$\deg'((\ell')^{d - 2q}  y_i'  y_j')$. 
For $j \le p$, since $q > 0$ we have $y_j'  y_{p+1}' = 0$ in $K[\Delta']$.  By Lemma~\ref{lem:deglocal}, $M'$ is a block diagonal matrix whose northwest $p \times p$ block is $M$. 
Observe that the closed star of the vertex $n + 1$ in $\Delta'$ is isomorphic to the closed star of the vertex $1$ in $S^{d - 1}$ in Example~\ref{ex:boundarysimplex}. Then Lemma~\ref{lem:deglocal} implies that after relabeling indices 
the  $(p+1, p+1)$ entry $M_{p+1,p+1}'$ of $M'$ is equal to the degree $\deg_{S^{d - 1}}$ of $\ell^{d - 2q}  x_{1}^{2q}$ in $S^{d - 1}$ (up to sign). 
By \eqref{eq:degx1dsphere}, $\ord_{[F]} (M_{p+1,p+1}') = \ord_{[\{ 2,\ldots,d + 1\}]} (\deg_{S^{d - 1}}(\ell^{d - 2q}  x_{1}^{2q})) = 2q - 1$. 
	
	We see that $M'$ is nonsingular, so 
	$\{y_1', \dotsc, y_{p+1}'\}$ 
	is a basis of $\overline{H}^{q}(\Delta')$. In particular, $\det M'$ computes the determinant of the Hodge--Riemann form. 
	
	As $F$ is not a facet of $\Delta'$, Theorem~\ref{thm:strongg} and Proposition~\ref{prop:mult0} give that $\ord_{[F]}(\det M')$ is even. 
	Since 
	$\ord_{[F]} (M_{p+1,p+1}') = 2q - 1$ is odd, we see that $\ord_{[F]}(\det M)$ is odd, as desired. 
\end{proof}

\medskip

\begin{proof}[Proof of Theorem~\ref{thm:middledegree}]
Note that for even $d$ and $q=d/2$, the above proof of Theorem~\ref{thm:oddmultHodgeRiemann} works over all characteristics because 
we are in the middle degree, and so the assumption that $\ell_F$ is a strong Lefschetz element holds vacuously. This case of
Theorem~\ref{thm:oddmultHodgeRiemann} implies Corollary~\ref{cor:ordervanishing}. 
Hence if $F$ is a facet of $\Delta$, then $\ord_{[F]}(D_{d/2}) = 1  \in \Z/2\Z$. 
Let $P \in k[a_{i,j}]$ be an irreducible polynomial that is not equal (up to multiplication by a scalar) to one of the polynomials
$\{ [F] : F \textrm{ facet of } \Delta \}$. Over $\overline{k}[a_{i,j}]$, we may factor $P = P_1^{m_1} \dotsb P_r^{m_r}$, where the $P_i$ are distinct irreducible polynomials over $\overline{k}$ and $m_i \in \Z_{>0}$. Note that none of the $P_i$ are scalar multiples of $[F]$.
We claim that there are monomials $y_1, \dotsc, y_p$ such that 
their images form a basis of 
$\overline{H}^{d/2}(\Delta)$ and $\ord_{P_1}(\det M) = 0$, where
$M$ is the $p \times p$ matrix whose $(i, j)$ entry is $\deg(y_i  y_j)$. 
This implies that $\ord_{P}(\det M) = 0$ and hence $\ord_P(D_{d/2}) = 0$.
We deduce that 
$D_{d/2} = \lambda \prod_{F \text{ facet of }\Delta} [F] \in K^{\times}/(K^{\times})^2$ for some $\lambda \in k^{\times}/(k^{\times})^2$, completing the proof.

It remains to verify the claim. 
Let $V(P_1)$ be the vanishing locus of $P_1$ inside $\mathbb{A}^{dn}_{\overline{k}}$, and let $(\mu_{i,j}) \in V(P_1)$ be a $\overline{k}$-point. 
Set $\mu_i = \sum_{j} \mu_{i,j} x_j$. First suppose that $\mu = (\mu_1, \dotsc, \mu_d)$ is an l.s.o.p. for $\overline{k}(a_{i,j})[\Delta]$.
Observe that $\ev_{\mu}(P_1) = 0$ since $(\mu_{i,j}) \in V(P_1)$. 
Then the claim follows from Lemma~\ref{lem:ordP}. 
Note that the assumption in Lemma~\ref{lem:ordP} that $\ell_\mu$ is a strong Lefschetz element holds vacuously since we are in middle degree. 
Hence we may assume that
$\mu$ is not an l.s.o.p. 
By Proposition~\ref{prop:stanleycriterion} there must be some facet $F$ of $\Delta$ such that $(\mu_{i,j})$ is contained in the vanishing locus of $[F]$. Applying this to every $\overline{k}$-point of $V(P_1)$, we see that 
$$V(P_1) \subset \bigcup_{F \text{ facet of }\Delta} V([F]).$$
As there are only finitely many facets, this implies that $V(P_1)$ is contained in $V([F])$ for some facet $F$. The irreducibility of $[F]$ then implies that $P_1$ and $[F]$ are equal up to multiplication by a scalar, a contradiction. 
\end{proof}

\section{Further discussion}\label{sec:discussion}

We conjecture an extension of Theorem~\ref{thm:oddmultHodgeRiemann} to pseudomanifolds in arbitrary characteristic. 

\begin{conjecture}\label{conj:pseudoextension}
	 Let $k$ be a field of arbitrary characteristic.
Let $\Delta$ be a connected oriented simplicial pseudomanifold over $k$ of dimension $d-1$ with vertex set $V$. For all $0 \le q \le d/2$, let $D_q$ be the determinant of the Hodge--Riemann form on $\overline{H}^q(\Delta)$. 	Let $F$ be a subset of $V$ of size $d$. Then
	$$\ord_{[F]}(D_{q}) = \begin{cases} 1 & \text{if }F \text{ is a facet of }\Delta \\ 0 & \text{otherwise}.\end{cases}$$
\end{conjecture}

When $\Delta$ is a $k$-homology manifold, the proof of Theorem~\ref{thm:oddmultHodgeRiemann} shows that if $\overline{H}_F(\Delta')$ has the strong Lefschetz property whenever $\Delta' = \Delta$ or $\Delta'$ is the stellar subdivision of $\Delta$ at the interior of a facet and $F$ is a non-face, then Conjecture~\ref{conj:pseudoextension} holds for $\Delta$. 

However, an additional ingredient is needed to establish Conjecture~\ref{conj:pseudoextension} for pseudomanifolds. A key property of homology manifolds which was used in the proofs of our results, e.g., Lemma~\ref{lem:ordP}, was that the dimension of $\overline{H}^q_{\mu}(\Delta)$ does not depend on $\mu$, the chosen l.s.o.p. (see Proposition~\ref{prop:samedim}). We show in the example below that this independence of the dimension can fail for pseudomanifolds.

\begin{example}\label{ex:hilbertdependence}
Let $\Delta$ be the standard $6$ vertex triangulation of $\mathbb{R} \mathbb{P}^2$, and let $\Delta' = \Delta * S^0$ be the suspension. Over a field of characteristic $2$, $\Delta'$ is an oriented pseudomanifold, but it is not a homology manifold. Using Macaulay2 \cite{M2}, we checked that, if one chooses an l.s.o.p. $\mu_1, \mu_2, \mu_3, \mu_4$ with all coefficients random elements of the field with $1024$ elements, the Hilbert function of $H_{\mu}(\Delta')$ is usually given by $(1, 4, 9, 6, 1)$, and the Hilbert function of $\overline{H}_{\mu}(\Delta')$ is usually given by $(1, 4, 8, 4, 1)$. If one chooses $\mu_1', \mu_2', \mu_3'$ to be generic linear combinations of the vertices of $\mathbb{R} \mathbb{P}^2$ and chooses $\mu_4'$ to be a generic linear combination of the vertices of $S^0$, then $H_{\mu'}(\Delta') = H_{(\mu_1', \mu_2', \mu_3')}(\Delta) \otimes H_{(\mu_4')}(S^0)$, and similarly for $\overline{H}_{\mu'}(\Delta')$. We can then use Proposition~\ref{prop:samedim} to compute that the Hilbert function of $H_{\mu'}(\Delta')$ is given by $(1, 4, 9, 7, 1)$, and the Hilbert function of $\overline{H}_{\mu'}(\Delta')$ is given by $(1, 4, 6, 4, 1)$. 
\end{example}

It follows from Theorem~\ref{thm:stronggpseudo} that the equality $\dim \overline{H}^q(\Delta) = \dim \overline{H}_F^q(\Delta)$ 
holding for all $\Delta$ would imply Conjecture~\ref{conj:pseudoextension} in the characteristic $2$ case. We now show that this equality holds when $q=1$ (in any characteristic).

\begin{proposition} \label{prop:degree1}
	 Let $k$ be a field of arbitrary characteristic.
Let $\Delta$ be a connected 
	oriented 
	simplicial pseudomanifold over $k$ of dimension $d-1$.
	Let $\mu$ be any l.s.o.p.~ for $K[\Delta]$. Then $\overline{H}^1_\mu(\Delta)=H^1_\mu(\Delta)$. In particular, $\dim\overline{H}^1_\mu(\Delta)=n-d$ independently of $\mu$. 
\end{proposition}

\begin{proof}
	Let $z=\sum_{i=1}^n \lambda_i x_i\in H^1_\mu(\Delta)$ be such that $zx_G=0$ for all codimension $1$ faces $G$ of $\Delta$. To prove the statement, we need to show that $z=0$.
	
	Let $F'$ be any facet of $\Delta$. Since $\mu$ is an l.s.o.p and $F'$ is a facet, by Proposition~\ref{prop:stanleycriterion}, 
	$\ev_\mu([F']) \neq 0$. It follows from Lemma~\ref{lem:Gaussianelim} 
	that $\{x_v : v\notin F'\}$ is a basis of $H^1_\mu(\Delta)$. By using this basis, we can assume $\lambda_i = 0$ for $i \in F'$.

	Consider any codimension $1$ face $G\subset F'$. Since $\Delta$ is a pseudomanifold, there is a unique facet $F''$ of $\Delta$ such that $F'\cap F''=G$; we let $u$ denote the unique vertex of $F''\smallsetminus G$. Then $0=zx_G=\lambda_u x_{F''}$. As $x_{F''}$ is nonzero because it has nonzero degree, it follows that $\lambda_u=0$. We conclude that $\lambda_v=0$ for all $v\in F''$. Continuing in this way, for any sequence of facets $F' = F_0, \ldots, F_r$ of $\Delta$ such that $F_i \cap F_{i + 1}$ is a face of codimension $1$ for $0 \le i < r$, we deduce that $\lambda_v=0$ for all $v \in F_r$. Since $\Delta$ is a connected pseudomanifold, every facet of $\Delta$ appears as $F_r$ for some such sequence, and hence $\lambda_v=0$ for all $v\in V$. Therefore $z=0$, as desired. 
\end{proof}

\begin{remark}\label{rem:q=1}
Proposition~\ref{prop:degree1} and Theorem~\ref{thm:stronggpseudo} imply that Conjecture~\ref{conj:pseudoextension} holds when $k$ has characteristic $2$ and $q=1$. By a specialization argument similar to the proof of Theorem~\ref{thm:strongg}, we see that Conjecture~\ref{conj:pseudoextension} also holds when $k$ has characteristic $0$ and $q=1$.
\end{remark}

\medskip

Assume that $\ell$ is a strong Lefschetz element in all degrees, i.e.,
the Hodge--Riemann form on $\overline{H}^q(\Delta)$ is nondegenerate for $0 \le q \le d/2$. The \emph{primitive part} of $\overline{H}^q(\Delta)$ is 
$\overline{H}^q_{\prim}(\Delta) := \{ y \in \overline{H}^q(\Delta) : \ell^{d - 2q + 1}  y = 0 \}$. 
Let $D_{\prim,q} \in K^{\times}/(K^{\times})^2$ be the determinant of the induced Hodge--Riemann form on $\overline{H}^q_{\prim}(\Delta)$.
For $0 < q \le d/2$, multiplication by $\ell$ induces an injection $\overline{H}^{q - 1}(\Delta) \to \overline{H}^q(\Delta)$ which splits to give an isomorphism 
$\overline{H}^q(\Delta) \cong \overline{H}^{q - 1}(\Delta) \oplus \overline{H}^q_{\prim}(\Delta)$. As this decomposition is orthogonal with respect to the Hodge--Riemann form, we have $D_q = D_{q - 1}D_{\prim,q}$. In particular, $D_q = D_0 \prod_{q' = 1}^q D_{\prim,q'}$. Since we established Conjecture~\ref{conj:pseudoextension} when $q = 0$ in Theorem~\ref{thm:i=0}, we conclude that 
Conjecture~\ref{conj:pseudoextension} holding for all $q$  is equivalent to the following conjecture.

\begin{conjecture}\label{conj:oddmultv2}
	 Let $k$ be a field of arbitrary characteristic.
	Let $\Delta$ be a connected oriented simplicial pseudomanifold over $k$ of dimension $d-1$ with vertex set $V$. 
	Then $\ell$ is a strong Lefschetz element in all degrees, and, for each subset $F$ of $V$ of size $d$ and 
	$0 < q \le d/2$, we have 
	$\ord_{[F]}(D_{\prim,q}) = 0$. 
\end{conjecture}

\begin{remark}
		Let $\Delta$ be a connected simplicial 
	pseudomanifold  
	of dimension $d-1$, and let $0 \le q \le d/2$. 
	Assume that $\Char k = 2$. By Theorem~\ref{thm:stronggpseudo},
	for every non-face $F$ of size $d$, $\overline{H}_F(\Delta)$ has the strong Lefschetz property in degree $q$. As above, the primitive part of $\overline{H}^q_F(\Delta)$ is 
	$\overline{H}^q_{F,\prim}(\Delta) := \{ y \in \overline{H}^q_F(\Delta) : \ell_F^{d - 2q + 1}  y = 0 \}$, and
	we have a decomposition $\overline{H}^q_F(\Delta) \cong \overline{H}^{q - 1}_F(\Delta) \oplus \overline{H}^q_{F,\prim}(\Delta)$. We claim that the restriction of the Hodge-Riemann form to $\overline{H}^q_{F,\prim}(\Delta)$ is anisotropic. 
	
	To prove the claim, we may assume that $q > 0$ and that the Hodge-Riemann form on $\overline{H}^q_F(\Delta)$ is not anisotropic. Recall that $W_q \subset \overline{H}_F^q(\Delta)$ is the span  of all monomials whose support is not contained in $F$.
	By Lemma~\ref{lem:codim1} and Corollary~\ref{cor:anisocriterion}, $W_q^\perp$ is $1$-dimensional, and $\ell_F^q$ is not in 
	$W_q = (W_q^\perp)^\perp$. That is, if
	$W_q^{\perp} = \spn(z)$, then $\deg_F(\ell_F^{d - 2q} z \ell_F^q)$ is nonzero. This implies that $\ell_F^{d - 2q + 1} z$ is nonzero and hence $z$ is not primitive. The claim now follows from Proposition~\ref{prop:almostanisotropy}.

\end{remark}

	\bibliography{Merged.bib}

\providecommand{\bysame}{\leavevmode\hbox to3em{\hrulefill}\thinspace}
\providecommand{\MR}{\relax\ifhmode\unskip\space\fi MR }
\providecommand{\MRhref}[2]{%
  \href{http://www.ams.org/mathscinet-getitem?mr=#1}{#2}
}
\providecommand{\href}[2]{#2}
\begin{thebibliography}{APP24}

\bibitem[Adi18]{Adiprasitog}
Karim Adiprasito, \emph{Combinatorial {L}efschetz theorems beyond positivity},
  ar{X}iv:1812.10454v4.

\bibitem[APP21]{APP}
Karim Adiprasito, Stavros~Argyrios Papadakis, and Vasiliki Petrotou,
  \emph{Anisotropy, biased pairings, and the {L}efschetz property for
  pseudomanifolds and cycles}, ar{X}iv:2101.07245v2.

\bibitem[APP24]{APPVolume}
\bysame, \emph{The volume intrinsic to a commutative graded algebra},
  ar{X}iv:2407.11916v1.

\bibitem[B{\c o}c64]{BocherIntroductionHigherAlgebra}
Maxime B{\c o}cher, \emph{Introduction to higher algebra}, Dover Publications,
  Inc., New York, 1964. \MR{172882}

\bibitem[Bri97]{BrionStructurePolytopeAlgebra}
Michel Brion, \emph{The structure of the polytope algebra}, Tohoku Math. J. (2)
  \textbf{49} (1997), no.~1, 1--32.

\bibitem[Gr{\"a}84]{Grabe}
Hans-Gert Gr{\"a}be, \emph{The canonical module of a {S}tanley-{R}eisner ring},
  J. Algebra \textbf{86} (1984), no.~1, 272--281.

\bibitem[GS]{M2}
Daniel~R. Grayson and Michael~E. Stillman, \emph{Macaulay2, a software system
  for research in algebraic geometry}, Available at
  \url{http://www2.macaulay2.com}.

\bibitem[KX23]{KaruXiaoAnisotropy}
Kalle Karu and Elizabeth Xiao, \emph{On the anisotropy theorem of {P}apadakis
  and {P}etrotou}, Algebr. Comb. \textbf{6} (2023), no.~5, 1313--1330.

\bibitem[Lee96]{Lee}
Carl~W. Lee, \emph{{P}.{L}.-spheres, convex polytopes, and stress}, Discrete
  Comput. Geom. \textbf{15} (1996), no.~4, 389–421.

\bibitem[MZ08]{MZ}
Juan Migliore and Fabrizio Zanello, \emph{The strength of the weak {L}efschetz
  property}, Illinois J. Math. \textbf{52} (2008), no.~4, 1417--1433.
  \MR{2595775}

\bibitem[NS09]{NovikSwartzGorensteinRings}
Isabella Novik and Ed~Swartz, \emph{Gorenstein rings through face rings of
  manifolds}, Compos. Math. \textbf{145} (2009), no.~4, 993--1000.

\bibitem[Oba25]{Oba}
Ryoshun Oba, \emph{Multigraded strong {L}efschetz property for balanced
  simplicial complexes}, Algebr. Comb. \textbf{8} (2025), no.~3, 775--794.
  \MR{4925165}

\bibitem[PP20]{PapadakisPetrotoug}
Stavros~Argyrios Papadakis and Vasiliki Petrotou, \emph{The characteristic 2
  anisotropicity of simplicial spheres}, ar{X}iv:2012.09815v1.

\bibitem[PP23]{PapadakisPetrotouGenericAnisotropy1Spheres}
\bysame, \emph{The generic anisotropy of simplicial 1-spheres}, Israel J. Math.
  \textbf{254} (2023), no.~1, 141--153.

\bibitem[Sch81]{SchenzelNumberFacesSimplicialComplexes}
Peter Schenzel, \emph{On the number of faces of simplicial complexes and the
  purity of {F}robenius}, Math. Z. \textbf{178} (1981), no.~1, 125--142.

\bibitem[Sta80]{Stanleyg}
Richard~P. Stanley, \emph{The number of faces of a simplicial convex polytope},
  Adv. in Math. \textbf{35} (1980), no.~3, 236--238. \MR{563925}

\bibitem[Sta96]{StanleyCombinatoricsCommutative}
\bysame, \emph{Combinatorics and commutative algebra}, second ed., Progress in
  Mathematics, vol.~41, Birkh\"auser Boston, Inc., Boston, MA, 1996.

\bibitem[Swa09]{Swartz09}
Ed~Swartz, \emph{Face enumeration---from spheres to manifolds}, J. Eur. Math.
  Soc. (JEMS) \textbf{11} (2009), no.~3, 449--485. \MR{2505437}

\bibitem[TW00]{TWHomological}
Tiong-Seng Tay and Walter Whiteley, \emph{A homological interpretation of
  skeletal ridigity}, Adv. in Appl. Math. \textbf{25} (2000), no.~1, 102--151.
  \MR{1773196}

\end{thebibliography}
	\bibliographystyle{amsalpha}

	\end{document}